\documentclass[review]{elsarticle}
\usepackage{amsmath}
\usepackage{graphicx,setspace}
\usepackage{amsfonts,amsmath, amssymb, amsthm}
\usepackage{mathrsfs}
\usepackage{epstopdf}
\usepackage{float}
\usepackage{lineno,hyperref}
\usepackage{color}
\usepackage{subfigure}
\usepackage{bm}
\usepackage{graphicx}
\usepackage{amssymb, amsthm}
\usepackage{mathrsfs}
\usepackage{epstopdf}
\usepackage{float}
\usepackage{lineno,hyperref}
\usepackage{color}
\usepackage{subfigure}
\usepackage{cases}
\usepackage{geometry}
\numberwithin{equation}{section}

\journal{}

\begin{document}

\newtheorem{definition}{Definition}
\newtheorem{lemma}{Lemma}
\newtheorem{remark}{Remark}
\newtheorem{theorem}{Theorem}
\newtheorem{proposition}{Proposition}
\newtheorem{assumption}{Assumption}
\newtheorem{example*}{Example}
\newtheorem{corollary}{Corollary}
\def\e{\varepsilon}
\def\Rn{\mathbb{R}^{n}}
\def\Rm{\mathbb{R}^{m}}
\def\E{\mathbb{E}}
\def\hte{\bar\theta}
\def\cC{{\mathcal C}}
\newcommand{\Cp}{\mathbb C}
\newcommand{\R}{\mathbb R}
\renewcommand{\L}{\mathcal L}
\newcommand{\supp}{\mathrm{supp}\,}

\numberwithin{equation}{section}

\begin{frontmatter}

\title{{\bf Random attractors for damped stochastic fractional Schr\"odinger equation on $\mathbb{R}^{n}$
}}
\author{\normalsize{\bf Li Lin$^{a,}\footnote{ linlimath@whut.edu.cn}$,
Yanjie Zhang$^{b}\footnote{zhangyj2022@zzu.edu.cn}$,
Ao Zhang$^{c,}\footnote{aozhang1993@csu.edu.cn}$,
} \\[10pt]
\footnotesize{${}^a$ School of Mathematics and Statistics, Wuhan University of Technology, Wuhan 430074, China} \\[5pt]
\footnotesize{${}^b$ School of Mathematics and Statistics, Zhengzhou University 450001,  China.}  \\[5pt]
\footnotesize{${}^c$
School of Mathematics and Statistics and HNP-LAMA,  Central South University,  Changsha 410083, China.}
}

\begin{abstract}
We study the random attractors associated with the stochastic fractional Schr\"odinger equation on $\mathbb{R}^n$. Utilizing the stochastic Strichartz estimates for the damped fractional Schr\"odinger equation with Gaussian noise, we show the existence and uniqueness of a global solution to the damped stochastic fractional nonlinear Schr\"odinger equation in $H^{\alpha}(\mathbb{R}^n)$. Furthermore, we demonstrate that this equation defines an infinite-dimensional dynamical system, which possesses a global attractor in $H^{\alpha}(\mathbb{R}^n)$.

\end{abstract}


\end{frontmatter}


\section{Introduction.}

The fractional nonlinear  Schr\"odinger equation appears in various fields such as nonlinear optics \cite{Longhi15}, quantum physics \cite{BM23} and water propagation \cite{Pusateri14}. Inspired by the Feynman path approach to quantum mechanics, Laskin \cite{Laskin02} used the path integral over L\'evy-like quantum mechanical paths to obtain a fractional Schr\"odinger equation.  Kirkpatrick et al. \cite{Kirkpatrick02} considered a general class of discrete nonlinear Schr\"odinger equations on the lattice $h\mathbb{Z}$ with mesh size $h>0$, they showed that the limiting dynamics were given by a nonlinear fractional  Schr\"odinger equation when $h\rightarrow 0$. Guo and Huang \cite{Guo12} applied concentration compactness and commutator estimates to obtain the existence of standing waves for nonlinear fractional Schr\"odinger equations under some assumptions. Shang and Zhang \cite{SZ} studied the existence and multiplicity of solutions for the critical fractional Schr\"odinger equation. They proved that the equation had a nonnegative ground state solution and also investigated the relation between the number of solutions and the topology of the set.  Choi and Aceves \cite{CA23} proved that the solutions of the discrete nonlinear Schr\"odinger equation with non-local algebraically decaying coupling converged strongly in $L^2(\mathbb{R}^n)$ to those of the continuum fractional nonlinear Schr\"odinger equation.  Frank and his collaborators \cite{LS15} proved general uniqueness results for radial solutions of linear and nonlinear equations involving the fractional Laplacian $(-\Delta)^{\alpha}$ with $\alpha\in(0,1)$ for any space dimensions $N\geq 1$. Wang and Huang \cite{Huang15} proposed an energy conservative difference scheme for the nonlinear fractional Schr\"odinger equations and gave a rigorous analysis of the conservation property. Boulenger et al. \cite{Boulenger16} derived a localized virial estimate for fractional nonlinear Schr\"odinger equation in $\mathbb{R}^{n}$ and proved the general blow-up result. Zhu and his collaborators \cite{Zhu16, Zhu18} found the sharp threshold mass of the existence of finite-time blow-up solutions and the sharp threshold of the scattering versus blow-up dichotomy for radial data. Lan \cite{LY22}  showed that if the initial data had negative energy and slightly supercritical mass, the solution for $L^2$-critical fractional Schr\"odinger equations blew up in finite time. Dinh \cite{Dinh19} studied dynamical properties of blow-up solutions to the focusing mass-critical nonlinear fractional Schr\"odinger equation, and obtained the $L^2$-concentration and the limiting profile with minimal mass of blow-up solutions. B. Alouini \cite{BA21, BA22} focused on the asymptotic dynamics for a dynamical system generated by a nonlinear dispersive fractional Schr\"odinger type equation and proved the existence of regular finite dimensional global attractor in the phase space with finite fractal dimension and the energy space. Goubet and Zahrouni \cite{GZ17} considered a weakly damped forced nonlinear fractional Schr\"odinger equation and proved that this global attractor had a finite fractal dimension if the external force was in a suitable weighted space.

In some circumstances, randomness has to be taken into account. The influence of noise on the propagation of waves is a very important problem. It can drastically change qualitative behaviors and result in new properties.  Bouard and Debussche \cite{Debussche99, Debussche05} studied a conservative stochastic nonlinear Schr\"odinger equation and the influence of multiplicative Gaussian noise, and showed the global existence and uniqueness of solutions. Herr et al. \cite{rockner19} studied the scattering behavior of global solutions for the stochastic nonlinear Schr\"odinger equations with linear multiplicative noise. Barbu and his collaborators \cite{Barbu17} showed the explosion could even be prevented with high probability on the whole time interval $[0, \infty)$. Debussche and Menza \cite{Debussche02} numerically investigated nonlinear Schr\"odinger equations with a stochastic contribution. Brze\'zniak et al. \cite{Liu21} established a new version of the stochastic Strichartz estimate for the stochastic convolution driven by jump noise. Deng et al. \cite{Deng22} studied the propagation of randomness under nonlinear dispersive equations by the theory of random tensors. Cui et al. \cite{H19} demonstrated that the solutions of stochastic nonlinear Schr\"odinger equations can be approximated by the solutions of coupled splitting systems. Yuan and Chen \cite{Chen17} proved the existence of martingale solutions for the stochastic fractional nonlinear Schr\"odinger equation on a bounded interval.  There are few results on the well-posedness on an unbounded interval and asymptotic behavior of damped stochastic fractional nonlinear Schr\"odinger equation in $H^{\alpha}(\mathbb{R}^n)$,  due to the complexity brought by the fractional Laplacian operator $(-\Delta)^{\alpha}$, the nonlinear term $|u|^{2\sigma}u$ and low regular of white noise.

In this paper, we examine the following stochastic fractional nonlinear Schr\"odinger equation with multiplicative noise in $H^{\alpha}(\mathbb{R}^n)$,
\begin{equation}
{\label{0lin}}
\left\{
\begin{aligned}
&i du-\left[(-\Delta)^{\alpha} u-|u|^{2\sigma}u\right] dt+i\gamma udt=f(t,x,u)dt- u \circ dW(t), \quad x \in \mathbb{R}^{n}, \quad t \geq 0, \\
&u(0)=u_0,
\end{aligned}
\right.
\end{equation}
where $u$ is a complex valued process defined on $\mathbb{R}^{n} \times \mathbb{R}^{+}$,  $0<\sigma,\gamma <\infty$,  and $W(t)$ is a  classical real-valued Wiener process. The fractional Laplacian operator $(-\Delta)^{\alpha}$ with  admissible exponent $\alpha \in (0,1)$ is involved.  The notation $\circ$ stands for Stratonovitch integral.

Let $(\Omega, \mathcal{F}, \mathbb{P})$ be a probability space, $\left(\mathcal{F}_{t}\right)_{t \geq 0}$ be a filtration, and let $\left(\beta_{k}\right)_{k \in \mathbb{N}}$ be a sequence of independent Brownian motions associated to this filtration.
The equation \eqref{0lin} can be rewritten as
\begin{equation}
i d u-\left[(-\Delta)^{\alpha} u-|u|^{2\sigma}u \right] d t+i\left(\gamma+\frac{1}{2} \right) udt=f(t,x,u)dt-udW.
\end{equation}

Here we are particularly interested in the long-term dynamics of the stochastic system \eqref{0lin}, especially the existence and uniqueness of random attractors for the damped fractional
Schr\"odinger equation with linear multiplicative noise. More precisely, we will prove the existence and uniqueness of the random attractor of the corresponding random dynamical system associated with \eqref{0lin} in $H^{\alpha}(\mathbb{R}^n)$. To study the random attractor of \eqref{0lin}, we must establish the well-posedness of the stochastic system \eqref{0lin}. In contrast to the case of stochastic nonlinear Schr\"odinger equation with $\alpha=2$. There are some essential difficulties in our problems. The first difficulty is the appearance of the fractional Laplacian operator $(-\Delta)^{\alpha}$. The second difficulty lies in how to solve the problem of low regularity of noisy sample paths and non-compact Sobolev embeddings on unbounded regions. The third difficulty lies in how to find the appropriate techniques for handling nonlinear terms $|u|^{2\sigma}u$, which is different from the existing results.

This paper is organized as follows.  In Section 2,  we introduce some notations and state our main results in this present paper.
In Section 3, we construct a truncated equation and prove the existence of a local solution of equation \eqref{0lin} in the whole space. Meanwhile, we use the stopping time technique, deterministic and stochastic fractional Strichartz inequalities to prove the global existence of the original equation \eqref{0lin}. In Section 4, we derive uniform estimates on the solution in $H^\alpha(\mathbb{R}^n)$ including the uniform estimates on the tails of solutions. Then, we prove the existence of random attractors in $H^\alpha(\mathbb{R}^n)$.

\section{Statement of the main results}
\subsection{Assumptions and preliminaries}
In this section, we borrow some well-known results from the theory of non-autonomous random dynamical systems. Let $(\Omega,\mathcal{F},\mathbb{P})$ be a probability space, $(X,\|\cdot\|_{X})$  be separable Banach spaces.
\begin{definition}
  Let $(\Omega,\mathcal{F},\mathbb{P},(\theta_t)_{t\in\mathbb{R}})$ be a metric dynamical systems. A mapping $\Phi:\mathbb{R^+}\times\mathbb{R}\times\Omega \times X\rightarrow X$ is called a continuous cocycle on $X$ over $(\Omega,\mathcal{F},\mathbb{P},(\theta_t)_{t\in\mathbb{R}})$ if for all
  $\tau \in \mathbb{R}, \omega\in \Omega$ and $t,s in \mathbb{R^+}$, the following conditions are satisfied:
\begin{itemize}
  \item $\Phi(\cdot,\tau,\cdot,\cdot):\mathbb{R^+}\times \Omega \times X\rightarrow X$ is a $(\mathcal{B}(R^+)\times\mathcal{F}\times \mathcal{B}(X)),\mathcal{B}(X))$-measurable mapping;
  \item $\Phi(0,\tau,\omega,\cdot)$ is the identity on $X$;
  \item $\Phi(t+s,\tau,\omega,\cdot)=\Phi(t,\tau+s,\theta_s\omega,\cdot)\circ
  \Phi(s,\tau,\omega,\cdot)$;
  \item $\Phi(t,\tau,\omega,\cdot):X\rightarrow X$ is continuous.
\end{itemize}
\end{definition}

\begin{definition}
    A random bounded set $\{D(\omega)\}_{\omega\in\Omega}\subset X$ is called
    tempered with respect to $(\theta_t)_{t\in \mathbb{R}}$ if
    for $P-a.e.~\omega\in\Omega$ and all $\varepsilon>0$
    $$\lim_{t\rightarrow \infty} e^{-\varepsilon t}d(D(\theta_{-t}\omega))=0,$$
    where $d(D)=sup_{\chi\in D}\|\chi\|_{X}.$
\end{definition}
 In the sequel, we will use $\mathcal{D}$ to denote the collection of all tempered families of bounded nonempty subsets of $H^{\alpha}(\mathbb{R}^n)$,

\begin{definition}
  Let  $\Phi$ be a continuous cocycle on $X$. Then $\Phi$ is said to be $\mathcal{D}$-pullback asymptotically compact in $X$ if for all
  $\omega\in \Omega, D\in \mathcal{D}$ and the sequences $t_n\rightarrow\infty,
  x_n\in D(-t_n,\theta_{-t_n}\omega)$, the sequence
$\{\Phi(t_n,-t_n,\theta_{-t_n}\omega,x_n)\}_{n=1}^\infty$ has a convergent
    subsequence in $X$.
\end{definition}

Now, we introduce some notations for clarity.  The notation $C(\cdot)$
represents a positive constant that depends on the parameter $\cdot$, its value may vary from one line to another. In cases involving more than two parameters, we use $C$ to denote the constant. For $p\geq 2$, the notation $L^{p}$ denotes the Lebesgue space of complex-valued functions.  The inner product in $L^{2}\left(\mathbb{R}^{n}\right)$ is endowed with
\begin{equation}
(f, g)={\bf{Re}} \int_{\mathbb{R}^{n}} f(x) \bar{g}(x) dx,
\end{equation}
for $f, g \in L^{2}\left(\mathbb{R}^{n}\right) $.

 The fractional Laplace operator $(-\Delta)^{\alpha}$ with $\alpha\in(0,1)$ is given by
\begin{equation*}
(-\Delta)^{\alpha}u(x)=C(n,\alpha) ~ \rm{P.V.} \int_{\mathbb{R}^n}\frac{u(x)-u(y)}{|x-y|^{n+2\alpha}}dy,\\
\end{equation*}
where $\rm{P.V.}$ means the principle value of the integral, and $C(n,\alpha)$ is a positive constant given by
\begin{equation}
\begin{aligned}
C(n,\alpha)=\frac{\alpha4^{\alpha}\Gamma(\frac{n+2\alpha}{2})}{\pi^{\frac{n}{2}}\Gamma(1-\alpha)}.
\end{aligned}
\end{equation}
For every $\alpha\in(0,1)$, the space $H^{\alpha}(\mathbb{R}^n)$ is denoted by
\begin{equation*}
\begin{aligned}
H^{\alpha}(\mathbb{R}^n)=\left\{u\in L^2(\mathbb{R}^n): \int_{\mathbb{R}^n}\int_{\mathbb{R}^n}\frac{|u(x)-u(y)|^2}{|x-y|^{n+2\alpha}}dxdy < \infty \right\}.
\end{aligned}
\end{equation*}
Then $H^{\alpha}(\mathbb{R}^n)$ is a Hilbert space with inner product given by
\begin{equation*}
(u,v)_{H^{\alpha}(\mathbb{R}^n)}=(u,v)_{L^{2}(\mathbb{R}^n)}+\mathbf{Re}\int_{\mathbb{R}^n}\int_{\mathbb{R}^n}\frac{\left(u(x)-u(y)\right)\left(\bar{v}(x)-\bar{v}(y)\right)}{|x-y|^{n+2\alpha}}dxdy
\end{equation*}
and
\begin{equation*}
(u,v)_{\dot{H}^{\alpha}(\mathbb{R}^n)}=\mathbf{Re}\int_{\mathbb{R}^n}\int_{\mathbb{R}^n}\frac{\left(u(x)-u(y)\right)\left(\bar{v}(x)-\bar{v}(y)\right)}{|x-y|^{n+2\alpha}}dxdy,\quad u, v\in H^{\alpha}(\mathbb{R}^n).
\end{equation*}
Then we have
\begin{equation*}
(u,v)_{H^{\alpha}(\mathbb{R}^n)}=(u,v)_{L^2(\mathbb{R}^n)}+(u,v)_{\dot{H}^{\alpha}(\mathbb{R}^n)},\quad u, v\in H^{\alpha}(\mathbb{R}^n).
\end{equation*}

By \cite{EG12}, we know that
\begin{equation*}
\|(-\Delta)^{\frac{\alpha}{2}}u\|^{2}_{L^{2}(\mathbb{R}^n)}=\frac{1}{2}C(n,\alpha)\|u\|^{2}_{\dot{H}^{\alpha}(\mathbb{R}^n)}.
\end{equation*}

Here, we give some regularity assumptions for the coefficient $f$ of system \eqref{0lin}. \\
\noindent ($\bf{A_{f}}$): For the function $f$ in \eqref{0lin}, we assume that $f: \mathbb{R}\times \mathbb{R}^n\times \mathbb{C}\rightarrow \mathbb{C}$ is continuous
and satisfies, for all $t, u\in \mathbb{R}$ and $x,y \in \mathbb{R}^{n}$,
\begin{align}
\label{assum1}
\mathbf{Im}\left(f(t,x,u)\bar{u}\right)&\leq -\beta|u|^2+\psi_1(t,x),\\
\label{assum2}
|f(t,x,u)|&\leq \psi_2(x)|u|+\psi_3(t,x),\\\label{assum3}
\frac{\partial f}{\partial u}(t,x,u)&\leq \psi_4(x),\\
\label{assum4}
|f(t,x,u)-f(t,y,u)|&\leq|\psi_5(x)-\psi_5(y)|,\\
\label{assum5}
\frac{\partial f_x}{\partial u}(t,x,u)&\leq\psi_6(t,x),\\
\label{assum6}
|f_x(t,x,u)-f_x(t,y,u)|&\leq|\psi_7(x)-\psi_7(y)|,\\
\end{align}
where $\beta>0$ is a constant, and
\begin{equation*}
\begin{aligned}
  &\psi_1 \in L^{\infty}_tL^{1}_x, \quad\psi_2\in L^{\infty}_x, \quad\psi_3\in L^{\infty}_tL^{2}_x \cap L^{\infty}_tL^{2\sigma+1}_x \cap L^{p}_tL^{\frac{n(\sigma+2)}{n+\sigma(n-\alpha)}}_x,  \\
  &\psi_4\in L^{\infty}_x, ~~\psi_5\in H^{\alpha}_x(\mathbb{R}^n),~~ \psi_6\in L^{\infty}_x, \quad\psi_7\in H^{\alpha}_x(\mathbb{R}^n).
\end{aligned}
\end{equation*}
with $p=\frac{4\alpha(\sigma+2)}{\sigma(n-2\alpha)}$.

\subsection{Main results}
This paper will study the weak pullback attractor for the damped stochastic fractional Schr\"odinger equation on $\mathbb{R}^n$ with $n\geq 2$. Our work is divided into two parts. In the first part, we will examine the well-posedness of a global solution for the systems with the fractional Laplacian operator, nonlinear term $|u|^{2\sigma}u$ and white noise. In the second part, we will prove that this equation provides an infinite dimensional dynamical system that possesses a random attractor in $H^{\alpha}(\mathbb{R}^n) $.

The first main result of this paper is about the existence and uniqueness of a global solution for the systems \eqref{0lin}.

\begin{theorem}
\label{theorem1}
Let $n \geq 2$, $\alpha$ be in the interval $\left(\frac{n}{2 n-1}, 1\right)$ and $0<\sigma<\frac{2 \alpha}{n-2 \alpha}$. Let
\begin{equation*}
p=\frac{4 \alpha(\sigma+2)}{\sigma(n-2 \alpha)}, \quad q=\frac{n(\sigma+2)}{n+\sigma \alpha}.
\end{equation*}
Then under Assumptions ($\bf{A_{f}}$) and $\sigma n=2\alpha$, for any $u_0 \in H^{\alpha}$ radial, there exists a unique global solution of equation \eqref{0lin} in $H^{\alpha}$, i.e., $\tau^{\ast}(u_0)=\infty$.
\end{theorem}
\begin{remark}
Theorem \ref{theorem1} implies that a global solution can only be achieved under the condition $\sigma n=2\alpha$ due to technical reasons. Obtaining global well-posedness without this assumption $\sigma n = 2\alpha$ remains an open problem.
\end{remark}
The following result gives.
\begin{theorem}
\label{theorem2}
Assume that $u_0 \in H^{\alpha}$ radial, Assumptions ($\bf{A_{f}}$)   and $\sigma n=2\alpha$ hold. Then the cocycle $\Phi$ of problem \eqref{0lin} has a unique $\mathcal{D}$-pullback randm attractor $\mathcal{A}=\{\mathcal{A}(\tau, \omega): \tau \in \mathbb{R}, \omega \in \Omega\}$ in
 $H^{s}(\mathbb{R}^n)$.
\end{theorem}

\section{Existence of global solution}
Let us first recall the definition of an admissible pair. We say a pair $(p, q)$ is admissible if
\begin{equation}
  p \in[2, \infty], \quad q \in[2, \infty), \quad(p, q) \neq\left(2, \frac{4 n-2}{2 n-3}\right), \quad \frac{2\alpha}{p}+\frac{n}{q}=\frac{n}{2}.
\end{equation}

In this subsection, we will use the stopping time technique, contraction mapping theorem in a suitable space, and conservation of mass to establish the existence of a global solution based on the stochastic fractional Strichartz estimates.

\subsection{Radial local theory}

The unitary group $S(t):=e^{-i t(-\Delta)^{\alpha}}$ enjoys several types of Strichartz estimates, for instance, non-radial Strichartz estimates, radial Strichartz estimates, and weighted Strichartz estimates. We only recall here radial Strichartz estimates (see, e.g., Ref. \cite{VDD18}).
\begin{lemma}
\label{0estimate}
(Radial Strichartz estimates) For $n \geq 2$ and $\frac{n}{2 n-1} \leq \alpha<1$, there exists a positive constant $C$ such that the following estimates hold:
  \begin{equation}\label{Strichartz}
  \begin{aligned}
  \left\|e^{-i t(-\Delta)^{\alpha}}\psi\right\|_{L^p\left(\mathbb{R}, L^q\right)} &\leq C\|\psi\|_{L^2},\\
  \left\|\int_{0}^{t} S(t-s) f(s) d s\right\|_{L^{p}\left(\mathbb{R}, L^{q}\right)} & \leq C \|f\|_{L^{\beta^{\prime}}\left(\mathbb{R}, L^{\gamma^{\prime}}\right)},
  \end{aligned}
  \end{equation}
  where $\psi$ and $f$ are radially symmetric, $(p, q)$ and $(\beta, \gamma)$ satisfy the fractional admissible condition, and $\frac{1}{\beta^{\prime}}+\frac{1}{\beta}=\frac{1}{\gamma^{\prime}}+\frac{1}{\gamma}=1$.
\end{lemma}

Next, we will present the local well-posedness of the stochastic fractional nonlinear Schr\"odinger equation \eqref{0lin} with radially symmetric initial data in the energy space $H^{\alpha}$.
\begin{proposition}(Radial local theory).
\label{theorem1}
 Let $n \geq 2$, $\alpha$ be in the interval $\left(\frac{n}{2 n-1}, 1\right)$ and $0<\sigma<\frac{2 \alpha}{n-2 \alpha}$. Let
\begin{equation*}
p=\frac{4 \alpha(\sigma+2)}{\sigma(n-2 \alpha)}, \quad q=\frac{n(\sigma+2)}{n+\sigma \alpha}. \end{equation*}
Then under Assumption ($\bf{A_{f}}$), for any $u_0 \in H^{\alpha}$ radial, there exists a stopping time $\tau^*\left(u_0, \omega\right)$ and a unique solution to \eqref{0lin} starting from $u_0$, which is almost surely in~ $C\left([0, \tau]; H^{\alpha}\right) \cap~L^p\left(0, \tau ; W^{\alpha, q}\right)$ for any $\tau<\tau^*\left(u_0\right)$. Moreover, we have almost surely,
\begin{equation*}
\tau^*\left(u_0, \omega\right)=+\infty \quad \text { or } \quad \limsup _{t\to\tau^*\left(u_0, \omega\right)}|u(t)|_{H^{\alpha}}=+\infty.
\end{equation*}
\end{proposition}

\begin{proof}
We follow the approach of \cite[Theorem 2.1]{Debussche99}. For the damping term $i\left(\gamma+\frac{1}{2}\right) u$, we can define a new free operator $T_{\gamma,\alpha}\phi:=e^{-\left(\gamma +\frac{1}{2}\right)t}\mathcal{F}^{-1}(e^{-it|y|^{2\alpha}})*\phi$. It has the same Strichartz estimates as the unitary group $S(t):=e^{-it(-\Delta)^{\alpha}}$ (see \cite{Sa15}, Proposition 2.7 and 2.9). So we can ignore the impact of the damping term and still use the unitary group $S(t)$ to represent the mild form of the equation \eqref{0lin} for simplicity.

Let $\theta \in \mathcal{C}_{0}^{\infty}(\mathbb{R})$ with $ \supp \theta \subset(-2,2), \theta(x)=1$ for $x \in[-1,1]$ and $0 \leq \theta(x) \leq 1$ for
$x \in \mathbb{R}$. Let $R>0$ and $\theta_{R}(x)=\theta\left(\frac{x}{R}\right)$. Let us fix $R\geq 1$.
We recall the mild form of equation \eqref{0lin}, that is
\begin{equation*}
\label{ut}
\begin{aligned}
u(t)=& S(t) u_{0}+i \int_{0}^{t} S(t-s)\left(\theta_{R}\left(|u|_{L^{r}\left(0, s ; L_{x}^{p}\right)}\right) |u(s)|^{2\sigma}u(s)\right) ds
+i \int_{0}^{t} S(t-s)u(s) d W(s)\\
&-i\int_{0}^{t} S(t-s)\left(\theta_{R}\left(|u|_{L^{r}\left(0, s ; L_{x}^{p}\right)}\right) f\right)ds.
\end{aligned}
\end{equation*}

It is easy to check that $(p, q)$ satisfies the fractional admissible condition. We choose $(a, b)$ so that
\begin{equation*}
\frac{1}{p^{\prime}}=\frac{1}{p}+\frac{2\sigma}{a}, \quad \frac{1}{q^{\prime}}=\frac{1}{q}+\frac{2\sigma}{b} .
\end{equation*}
We see that
\begin{equation*}
\frac{2\sigma}{a}-\frac{2\sigma}{p}=1-\frac{\sigma(n-2 \alpha)(\sigma+1)}{2 \alpha(\sigma+2)}=: \theta>0, \quad q \leq b=\frac{n q}{n-\alpha q},\quad q<\frac{2 n}{n-\alpha}.
\end{equation*}
The latter gives the Sobolev embedding $W^{\alpha, q} \hookrightarrow L^b$. Define the spaces
\begin{equation*}
\mathcal{X}:=\left\{C\left(I, H^\alpha\right) \cap L^p\left(I, W^{\alpha, q}\right):\|u\|_{L^{\infty}\left(I, H^\alpha\right)}+\|u\|_{L^p\left(I, W^{\alpha, q}\right)} \leq M\right\},
\end{equation*}
where $I=[0, \zeta]$ and $M, \zeta>0$ to be chosen later. Then the proof is based on a fixed point argument. Let us briefly discuss the proof, more details are similar to Theorem 1 in \cite{AZ24}. Here we only need to prove that the different term $i\int_{0}^{t} S(t-s)\left(\theta_{R}\left(|u|_{\mathcal{X}}\right) f\right)ds$ belongs to the space $\mathcal{X}$. For the term $i\int_{0}^{t} S(t-s)\left(\theta_{R}\left(|u|_{\mathcal{X}}\right) f\right)ds$, using the Strichartz estimates,  Assumption ($\bf{A_{f}}$) and H\"older inequality, one has
\begin{equation}
\begin{aligned}
  &\left\|i\int_{0}^{t} S(t-s)\left(\theta_{R}\left(|u|_{X}\right) f\right)ds\right\|_{{L^p\left(I, W^{\alpha, q}\right)}}\\
  &\leq C\left(\theta_{R}\left(|u|_{X}\right)f\right)_{{L^{p^{'}}\left(I, W^{\alpha, {q^{'}}}\right)}}\\
  &\leq C_{R}\left\|f\right\|_{{L^{p^{'}}\left(I, W^{\alpha, {q^{'}}}\right)}}\\
  &\leq C_{R}\|{\psi_2(x)}|u|\|_{{L^{p^{'}}\left(I, W^{\alpha, {q^{'}}}\right)}}+C_{R}\|\psi_3(t,x)\|_{{L^{p^{'}}\left(I, W^{\alpha, {q^{'}}}\right)}} \\
  &\leq CI^{1-2/p}\left(\|\psi_2(x)\|_{W^{\alpha,\frac{q}{q-2}}}\|u\|_{{L^p\left(I, W^{\alpha, q}\right)}}+\|\psi_3(t,x)\|_{{L^p\left(I, W^{\alpha, \frac{n(\sigma+2)}{n+\sigma(n-\alpha)}}\right)}}\right).
\end{aligned}
\end{equation}
Similarly, we can also prove that the term $i\int_{0}^{t} S(t-s)\left(\theta_{R}\left(|u|_{X}\right) f\right)ds$ belongs to the space ${L^{\infty}\left(I, H^{\alpha}\right)}$. \\
Next, by the Strichartz estimates, Assumption ($\bf{A_{f}}$) and H\"older inequality again, we get
\begin{equation}
\begin{aligned}
&\left|\int_{0}^{t} S(t-s)(\theta_{R}\left(|u_1|_{X}\right)f(x,u_1)-\theta_{R}\left(|u_2|_{X}\right)f(x,u_2))ds\right|_{{L^p\left(I, W^{\alpha, q}\right)}}\\
&\leq |\theta_{R}\left(|u_1|_{X}\right)f(x,u_1)-\theta_{R}\left(|u_2|_{X}\right)f(x,u_2)|_{{L^{p^{'}}\left(I, W^{\alpha, {q^{'}}}\right)}}\\
&\leq |\left(\theta_{R}\left(|u_1|_{X}\right)-\theta_{R}\left(|u_2|_{X}\right)\right)f(x,u_2)|_{{L^{p^{'}}\left(I, W^{\alpha, {q^{'}}}\right)}}\\
&\quad +|\theta_{R}\left(|u_1|_{X}\right)\left(f(x,u_1)-f(x,u_2)\right)|_{{L^{p^{'}}\left(I, W^{\alpha, {q^{'}}}\right)}} \\
&\leq C_R\|u_1-u_2\|_{X}\|f(x,u_2)\|_{{L^{p^{'}}\left(I, W^{\alpha, {q^{'}}}\right)}}\\
&\quad +C_{R}\|f(x,u_1)-f(x,u_2)\|_{{L^{p^{'}}\left(I, W^{\alpha, {q^{'}}}\right)}}\\
&\leq  C_{R}I^{1-2/p}\left(\|\psi_2(x)\|_{W^{\alpha,\frac{q}{q-2}}}\|u\|_{{L^p\left(I, W^{\alpha, q}\right)}}+\|\psi_3(t,x)\|_{{L^p\left(I, W^{\alpha, q^{'}}\right)}}\right)\\
&\quad +C_{R}\|\psi_4(x)(u_1-u_2)\|_{L^{r^{\prime}}\left(0, T ; L^{p^{\prime}}(\mathbb{R}^n)\right)}\\
&\leq C_{R}I^{1-2/p}\left(\|\psi_2(x)\|_{W^{\alpha,\frac{q}{q-2}}}\|u\|_{{L^p\left(I, W^{\alpha, q}\right)}}+\|\psi_3(t,x)\|_{{L^p\left(I, W^{\alpha, \frac{n(\sigma+2)}{n+\sigma(n-\alpha)}}\right)}}\right)\\&+C_{R}I^{1-2/p}\left(\|\psi_4(x)\|_{W^{\alpha,\frac{q}{q-2}}}\|u_1-u_2\|_{{L^p\left(I, W^{\alpha, \frac{n(\sigma+2)}{n+\sigma(n-\alpha)}}\right)}}\right).
\end{aligned}
\end{equation}
The estimate in ${L^{\infty}\left(I, H^{\alpha}\right)}$ is similar.
\end{proof}

The fractional nonlinearity Schr\"odinger shares  the similarity with the classical nonlinear Schr\"odinger equation, which has the formal law for the mass and energy by
\begin{equation}
\begin{aligned}
\label{0Energy0}
M[u]&:=\|u\|^{2}_{L^2}, \\
H[u]&:=\frac{1}{2}\int_{\mathbb{R}^n} \left|(-\Delta)^{\frac{\alpha}{2}}u\right|^2dx -\frac{1}{2\sigma +2}\int_{\mathbb{R}^n} |u|^{2\sigma+2}dx.
\end{aligned}
\end{equation}

In the following, we will give the results about the energy $H[u]$.
\begin{proposition}
\label{energy}
Assume that $u_0 \in H^{\alpha}$ radial and Assumption ($\bf{A_{f}}$) holds. For any stopping time $\tau<\tau^{*}(u_0)$, we have
 \begin{equation}
 \label{Mass}
 M(u(\tau))\leq C\left(u_0, \psi_1\right)e^{-2(\gamma+\beta)\tau}, ~~a.s.
 \end{equation}
 and
\begin{equation}\label{0energy1}
\begin{aligned}
H(u(\tau))&=H(u_0)-\mathbf{Im} \int^{\tau}_0 \int_{\mathbb{R}^n} (-\Delta)^{\alpha}\bar{u}udxdW
-\gamma\mathbf{Re}\int^{\tau}_0 \int_{\mathbb{R}^n}\left[(-\Delta)^{\alpha}\bar{u}\right]udxdt
\\
&+\gamma \int^{\tau}_0\int_{\mathbb{R}^n}|u|^{2\sigma+2}dxdt+\mathbf{Im}\int^{\tau}_0\int_{\mathbb{R}^n}\left[(-\Delta)^{\alpha}\bar{u}-|u|^{2\sigma}\bar{u}\right]fdxdt
\end{aligned}
\end{equation}
\end{proposition}
\begin{proof}
Note that the equation \eqref{0lin} can be rewritten as
\begin{equation}
\begin{aligned}
du&=-i\left[(-\Delta)^{\alpha}u-|u|^{2\sigma}u\right]dt-\left(\gamma+\frac{1}{2}\right) udt -if(t,x,u)dt-iudW(t).
\end{aligned}
\end{equation}
We can conclude that
\begin{equation}
\begin{aligned}
M^{'}[u]h&=2\mathbf{Re} (u,h),\\
M^{''}(h,k)&=2\mathbf{Re}(k,h).\\
\end{aligned}
\end{equation}
\begin{equation}
\begin{aligned}
H^{'}[u]h&=\mathbf{Re} \int_{\mathbb{R}^n} (-\Delta)^{\frac{\alpha}{2}}\bar{u}(-\Delta)^{\frac{\alpha}{2}}{h}dx-\mathbf{Re}\int_{\mathbb{R}^n} |u|^{2\sigma}\bar{u}hdx
=\mathbf{Re} \int_{\mathbb{R}^n}[(-\Delta)^{\alpha}\bar{u}-|u|^{2\sigma}\bar{u}]hdx,
\end{aligned}
\end{equation}
and
\begin{equation}
\begin{aligned}
H^{''}(u)(h, k)=\mathbf{Re} \int_{\mathbb{R}^n} (-\Delta)^{\frac{\alpha}{2}}\bar{k}(-\Delta)^{\frac{\alpha}{2}}hdx
-\mathbf{ Re} \int_{\mathbb{R}^n} |u|^{2\sigma}\bar{k} hdx-2\sigma \int_{\mathbb{R}^n} |u|^{2\sigma-2}\mathbf{Re}(\bar{u}k)\mathbf{Re}(\bar{u}h)dx.
\end{aligned}
\end{equation}
Using It\^o's formula, we obtain
\begin{equation}
\label{mass}
\begin{aligned}
M(u(t))=M(u_0)+2\mathbf{{Im}}\int^{t}_{0}\left(f(s,x,u),u\right)ds-2\gamma \int^{t}_{0}M(u(s))ds,
\end{aligned}
\end{equation}
and
\begin{equation*}
\begin{aligned}
H(u(\tau))&=H(u_0)-\mathbf{Im} \int^{\tau}_0 \int_{\mathbb{R}^n} (-\Delta)^{\alpha}\bar{u}udxdW
-\gamma\mathbf{Re}\int^{\tau}_0 \int_{\mathbb{R}^n}\left[(-\Delta)^{\alpha}\bar{u}\right]udxdt
\\
&+\gamma \int^{\tau}_0\int_{\mathbb{R}^n}|u|^{2\sigma+2}dxdt+\mathbf{Im}\int^{\tau}_0\int_{\mathbb{R}^n}\left[(-\Delta)^{\alpha}\bar{u}-|u|^{2\sigma}\bar{u}\right]fdxdt
\end{aligned}
\end{equation*}
By the equality \eqref{mass} and Assumption ($\bf{A_{f}}$), we have
\begin{equation}
\begin{aligned}
M(u(\tau))\leq M(u_0)-2(\gamma+\beta)\int^{\tau}_{0}M(u(s))ds+\|\psi_1(\cdot,\cdot)\|_{L^{\infty}_{t}L^{1}_{x}}.
\end{aligned}
\end{equation}
By Gr\"onwall inequality, we have
\begin{equation*}
M(u(\tau))\leq C\left(u_0, \psi_1\right)e^{-2(\gamma+\beta)\tau}.
\end{equation*}
\end{proof}

\subsection {Proof of Theorem \ref{theorem1}}
Proof of Theorem  \ref{theorem1}.
To ensure the global existence of the solution, we need to establish the uniform boundedness of $\|u\|_{H^{\alpha}}$. We first obtain the uniform boundedness of the energy $H[u]$. Then the energy evolution of $u$ implies that for any $T_0>0$, any stopping time $\tau<\inf(T_0, \tau^*(u_0))$ and any time $t\leq \tau$,
\begin{equation*}
\begin{aligned}
\mathbb{E}\left[\sup_{t\leq \tau}H(u(t))\right]&\leq \mathbb{E}\left[H(u_0)\right]+\mathbb{E}\left[\sup_{t\leq \tau} \left|\int^{t}_{0} \int_{\mathbb{R}^n}(-\Delta)^{\alpha}\bar{u}udxdW\right|\right]\\
&-\gamma \mathbb{E}\left[\sup_{t\leq \tau}\int^{t}_0 \int_{\mathbb{R}^n}\left[(-\Delta)^{\alpha}\bar{u}\right]udxds\right]+\gamma \mathbb{E}\left[\int^{\tau}_0\int_{\mathbb{R}^n}|u|^{2\sigma+2}dxdt\right]\\
&+\mathbb{E}\left[\mathbf{Im}\int^{\tau}_0\int_{\mathbb{R}^n}\left[(-\Delta)^{\alpha}\bar{u}-|u|^{2\sigma}\bar{u}\right]fdxdt\right]\\
&:=H(u_0)+\mathcal{J}_1+\mathcal{J}_2+\mathcal{J}_3.
\end{aligned}
\end{equation*}

 Treating the case where $\sigma=\frac{2\alpha}{n}$, we can utilize Gagliardo--Nirenberg's inequality (see \cite[Theorem 3.2]{SH17}) to obtain the following equation:
\begin{equation}
\label{in6}
\begin{aligned}
\|u\|^{2\sigma+2}_{L^{2\sigma+2}}&\leq C_{opt} \|u\|^{\frac{n\sigma}{\alpha}}_{\dot{H}^{\alpha}}\cdot \|u\|^{2\sigma+2-\frac{n\sigma}{\alpha}}_{L^2},
\end{aligned}
\end{equation}
where the best constant is given by
\begin{equation*}
C_{opt}=\frac{\sigma+1}{\|R\|^{2\sigma}},
\end{equation*}
and $R$ is a ground state of
\begin{equation*}
(-\Delta)^{\alpha}R+R-|R|^{2\sigma}R=0, ~~R\in H^{\alpha}.
\end{equation*}

By the inequalities \eqref{0Energy0} and \eqref{in6}, we know
\begin{equation}
\label{erb}
\begin{aligned}
\|u\|^2_{\dot{H}^{\alpha}}&=2H(u)+\frac{1}{\sigma+1}\|u\|^{2\sigma+2}_{L^{2\sigma+2}}
\leq 2H(u)+\frac{\|u\|^{2\sigma}}{\|R\|^{2\sigma}}\|u\|^{2\sigma}_{\dot{H}^{\alpha}}.
\end{aligned}
\end{equation}
Then  the inequality \eqref{erb} and the energy evolution of $u$ imply that for any $T_0>0$, any stopping time $\tau <\inf(T_0, \tau^{*}(u_0))$ and any time $t\leq \tau$.
\begin{equation}
\begin{aligned}
\mathbb{E}\left[\left(1-\frac{\|u(t)\|^{2\sigma}}{\|R\|^{2\sigma}}\right)\|u\|^2_{\dot{H}^{\alpha}}\right]\leq 2\mathbb{E}\left[H(u)\right].
\end{aligned}
\end{equation}

For the term $\mathcal{J}_1$, using the similar technique as
\cite[Theorem 2]{AZ24}, we have
\begin{equation}
\label{J1}
\begin{aligned}
\mathcal{J}_1&:=\mathbb{E}\left[\sup_{t\leq \tau} \left|\int^{t}_{0} \int_{\mathbb{R}^n}(-\Delta)^{\alpha}\bar{u}udxdW\right|\right]\\
&\leq C(n,\alpha,T_0,M(u_0))+C(n,\alpha,T_0,M(u_0))\varepsilon\mathbb{E}\left[\sup_{t\leq \tau}\|(-\Delta)^\frac{\alpha}{2}u\|^2\right]\\
&\quad+C(n,\alpha,T_0,M(u_0), \varepsilon)\int^{\tau}_{0}\mathbb{E}\left(\sup_{r\leq s} \|(-\Delta)^\frac{\alpha}{2}u(r)\|^2\right)ds.
\end{aligned}
\end{equation}
For the term $\mathcal{J}_2$, by the inequalities \eqref{in6} and \eqref{Mass}, we have
\begin{equation}
\label{J3}
\begin{aligned}
\mathcal{J}_2&:=-\gamma \mathbb{E}\left[\sup_{t\leq \tau}\int^{t}_0 \int_{\mathbb{R}^n}\left[(-\Delta)^{\alpha}\bar{u}\right]udxds\right]+\gamma \mathbb{E}\left[\int^{\tau}_0\int_{\mathbb{R}^n}|u|^{2\sigma+2}dxdt\right]\\
&=-\gamma \mathbb{E}\int^{t}_0\left[\sup_{t\leq \tau}\|u\|^2_{\dot{H}^{\alpha}}\right]ds+\gamma \int^{\tau}_0\int_{\mathbb{R}^n}\mathbb{E}\left[|u|^{2\sigma+2}\right]dxdt\\
&\leq \gamma \cdot C_{opt}\int^{\tau}_{0}e^{-(\gamma+\beta)\sigma s}\mathbb{E}\left[\|u\|^2_{\dot{H}^{\alpha}}\right]\|u_0\|^{2\sigma}ds\\
& \leq C\left(\gamma, C_{opt}, \beta, M(u_0)\right)\int^{\tau}_{0}\mathbb{E}[\sup_{r\leq s}\|u\|^2_{\dot{H}^{\alpha}}]ds.
\end{aligned}
\end{equation}
For the term $\mathcal{J}_3$, we have
\begin{equation*}
\begin{aligned}
\mathcal{J}_3&=\mathbb{E}\left[\mathbf{Im}\int^{\tau}_0\int_{\mathbb{R}^n}\left[(-\Delta)^{\alpha}\bar{u}\right]fdxdt-\mathbf{Im}\int^{\tau}_0\int_{\mathbb{R}^n}|u|^{2\sigma}\bar{u}fdxdt\right]
:=\mathcal{J}_{31}+\mathcal{J}_{32}.
\end{aligned}
\end{equation*}
For the term $\mathcal{J}_{31}$, under Assumption ($\bf{A_{f}}$), we have
\begin{equation*}
\begin{aligned}
\mathcal{J}_{31}&= \mathbb{E}\left[\mathbf{Im}\int^{\tau}_0\int_{\mathbb{R}^n}\left[(-\Delta)^{\alpha}\bar{u}\right]fdxdt\right]\\
&= \mathbb{E}\left[\mathbf{Im}\int^{\tau}_0\int_{\mathbb{R}^n}(-\Delta)^{\frac{\alpha}{2}}\bar{u}(-\Delta)^{\frac{\alpha}{2}}fdxdt\right]\\
&=C(n,\alpha)\mathbb{E}\left[\mathbf{Im}\int^{\tau}_0\int_{\mathbb{R}^n}\frac
 {\left[f(t,x,{u}(x))-f(t,y,{u}(y))\right]{(\bar{u}(x)-\bar{u}(y))}}{|x-y|^{n+2\alpha}}dxdyds\right]\\
 &=C(n,\alpha)\mathbb{E}\left[\mathbf{Im}\int^{\tau}_0\int_{\mathbb{R}^n}\frac
 {\left[f(t,x,{u}(x))-f(t,y,{u}(x))\right]{(\bar{u}(x)-\bar{u}(y))}}{|x-y|^{n+2\alpha}}dxdyds\right]\\
 &+C(n,\alpha)\mathbb{E}\left[\mathbf{Im}\int^{\tau}_0\int_{\mathbb{R}^n}\frac
 {\left[f(t,y,u(x))-f(t,y,u(y))\right]{(\bar{u}(x)-\bar{u}(y))}}{|x-y|^{n+2\alpha}}dxdyds\right]\\
 &\leq C(n,\alpha)\left\{\int^{\tau}_0\int_{\mathbb{R}^n}\frac
 {|\psi_5(x)-\psi_5(y)||\bar{u}(x)-\bar{u}(y)|}{|x-y|^{n+2\alpha}}dxdyds+
 \int^{\tau}_0\int_{\mathbb{R}^n}\psi_4(t,x)\frac
 {|\bar{u}(x)-\bar{u}(y)|^2}{|x-y|^{n+2\alpha}}dxdyds\right\}\\
 &\leq C(n,\alpha)\int^{\tau}_0\mathbb{E}\left[\|(-\Delta)^{\alpha/2}u\|^2\right]+C(n,\alpha, \psi_4).
\end{aligned}
\end{equation*}

For the term $\mathcal{J}_{32}$, by Assumption ($\bf{A_{f}}$), Young's inequality, along with the same technique as used in inequality \eqref{J3}, we have
\begin{equation*}
\begin{aligned}
\mathcal{J}_{32}:=\mathbb{E}\left[\left|-\mathbf{Im}\int^{\tau}_0\int_{\mathbb{R}^n}|u|^{2\sigma}\bar{u}fdxdt\right|\right]&\leq \mathbb{E}\left[\int^{\tau}_0\int_{\mathbb{R}^n}|u|^{2\sigma}|\bar{u}||f|dxdt\right]\\
& \leq \mathbb{E}\left[\int^{\tau}_0\int_{\mathbb{R}^n}|u|^{2\sigma+2}\psi_2(x)dxdt\right]+\mathbb{E}\left[\int^{\tau}_0\int_{\mathbb{R}^n}|u|^{2\sigma+1}\psi_3(t,x)dxdt\right]\\
& \leq |\psi_2(x)|_{L^{\infty}_x}\mathbb{E}\left[\int^{\tau}_0\int_{\mathbb{R}^n}|u|^{2\sigma+2}dxdt\right]+\mathbb{E}\left[\int^{\tau}_0\int_{\mathbb{R}^n}|u|^{2\sigma+1}\psi_3(t,x)dxdt\right]\\
& \leq |\psi_2(x)|_{L^{\infty}_x}\int^{\tau}_0\int_{\mathbb{R}^n}|u|^{2\sigma+2}dxdt+C_{\sigma}\int^{\tau}_0\int_{\mathbb{R}^n}|u|^{2\sigma+2}dxdt\\
&+C_{\sigma}\int^{\tau}_0\int_{\mathbb{R}^n}|\psi_3(t,x)|^{2\sigma+1}dxdt.\\
& \leq  C\left(\psi_2, \sigma, C_{opt},\beta, M(u_0)\right)\int^{\tau}_{0}\mathbb{E}[\sup_{r\leq s}\|u\|^2_{\dot{H}^{\alpha}}]ds+C(\sigma, \psi_3).
\end{aligned}
\end{equation*}
Therefore, we have
\begin{equation}
\label{J3}
\mathcal{J}_3 \leq C(\sigma, \psi_3, \psi_4, n, \alpha)+ C\left(\psi_2, \sigma, C_{opt},\beta, M(u_0), n, \alpha\right)\int^{\tau}_{0}\mathbb{E}[\sup_{r\leq s}\|u\|^2_{\dot{H}^{\alpha}}]ds.
\end{equation}
Combining the inequalities \eqref{J1} and \eqref{J3}, and using Gr\"onwall inequality, we obtain
\begin{equation*}
\mathbb{E}\sup_{t\leq \tau}\|u\|^2_{\dot{H}^{\alpha}} \leq C\left( n,\alpha,T_0, M(u_0), C_{opt},\beta,\psi_2,\psi_3,\psi_4\right).
\end{equation*}
These are priori estimates, combined with local well-posedness, imply the global existence of a unique solution.

\section{Existence of random attractors in $H^{\alpha}(\mathbb{R}^n)$}
This section is devoted to
show that the solution to a specific class of damped stochastic fractional Schr\"odinger equation defines an infinite-dimensional dynamical system, which possesses a global attractor in
$H^{\alpha}(\mathbb{R}^n)$.

\subsection{ Uniform tail estimates}

The uniform estimates of the tails of solutions are essential for proving the asymptotic compactness of the solution operators in $H^\alpha(\mathbb{R}^n)$.

For $t\in \mathbb{R}$, let
$\theta_t: \Omega \rightarrow \Omega$ be the classical transformation given by
\begin{equation}
\theta_t\omega(\cdot)=
\omega(\cdot+t)-\omega(t), ~~\omega \in \Omega.
\end{equation}
Let $z: \Omega \rightarrow \mathbb{R}$ be defined by
$z(\omega)=-\int^{0}_{-\infty}e^{\tau}\omega(\tau)d\tau$.
Then we know that
$y(t, \omega) = z(\theta_t \omega)$ is the unique stationary solution of the following equation
\begin{equation}
dy=-ydt+dW(t).
\end{equation}

Let $v(t,0,\omega,v_0):=e^{-i z(\theta_t\omega)}u(t,0,\omega,u_0)=e^{-i z(\theta_t\omega)}u(t,\omega,u_0)$. Then the equation \eqref{0lin} can be rewritten as
\begin{equation}
\label{lin5}
\left\{
\begin{aligned}
&dv+i(-\Delta)^{\alpha} v dt+\gamma vdt=ivz(\theta_t\omega)+i|v|^{2\sigma}v-if(t,x,e^{i z(\theta_t\omega)}v)e^{-i z(\theta_t\omega)}, \\
&v(0,x)=v_0(x).
\end{aligned}
\right.
\end{equation}
 Define a continuous cocycle $\Phi$ for problem (\ref{lin5}), which is given by $\Phi:\mathbb{R}^{+}\times \mathbb{R}\times \Omega \times H^{\alpha}(\mathbb{R}^n)\rightarrow H^{\alpha}(\mathbb{R}^n)$ and for every
$t\in \mathbb{R}^{+}, \tau \in \mathbb{R},  \omega\in \Omega$ and
$u_{\tau}\in H^{\alpha}(\mathbb{R}^n)$,
$\Phi(t,\tau,\omega, u_{\tau})=u(t+\tau,\tau, \theta_{-\tau}\omega, u_{\tau})=e^{i z(\theta_t\omega)}v(t+\tau,\tau,\theta_{-\tau}\omega,v_{\tau})$.

Next, we give the uniform estimates of solutions in $L^2({\mathbb{R}}^n).$
\begin{lemma}\label{linli1}
Let $u_0$, $\alpha$, $\sigma$ and $f$ be as in Theorem \ref{theorem1}. Then for every $\sigma_0\in\mathbb{R},
\omega\in \Omega$ and $D=\{D(\omega):
\omega\in \Omega\}\in \mathcal{D}$, there exists a positive constant $T=T(\omega,D,\sigma_0,\alpha)$
such that for all $t\geq T$, the solution $v$ of system  (\ref{lin5}) satisfies
\begin{align}
\begin{split}
||v(\sigma_0,-t,\omega,v_{-t})||^2&+\left(\frac{3}{4}\gamma+2\beta\right)\int_{-t}^{\sigma_0} e^{\frac{5}{4}
\gamma(s-\sigma_0)}||v(s,-t,\omega,v_{-t})||^2ds
\leq C+C\int_{-\infty}^{\sigma_0} e^{\frac{5}{4}
\gamma(s-\sigma_0)}||\psi_1(s, \cdot)||_{L^1_x}ds,
\end{split}
\end{align}
where $v_{-t}\in D(-t,\theta_{-t}\omega)$.
\begin{proof}
  Taking the inner product of equation (\ref{lin5}) with $v$ in the space  $L_2(\mathbb{R}^n)$, we obtain
  \begin{align}
\begin{split}
\frac{d}{dt}&\|v(t,\omega,v_0)\|^2+iC(n,s)\|v\|_{\dot{H}^\alpha}^2+2\gamma\|v\|^2
=2iz(\theta_{t}\omega)\|v\|^2+2i\|v\|_{L^{2\sigma+2}}^{2\sigma+2}-2ie^{-iz}\int_{\mathbb{R}^n} f\overline{v}dx.
\end{split}
\end{align}
Then taking the real part, we have
$$\frac{d}{dt}||v(t,\omega,v_0)||^2+2\gamma||v||^2=2\mathbf{Im}\left(e^{-iz}\int_{\mathbb{R}^n} f\overline{v}dx\right).$$
Under Assumption($\bf{A_{f}}$), we have
\begin{align}
\begin{split}
2\mathbf{Im}\left(e^{-iz}\int_{\mathbb{R}^n} f\overline{v}dx\right)&=2\mathbf{Im}\left(\int_{\mathbb{R}^n} f(t,x,e^{iz(\theta_t\omega)}v)\overline{e^{iz(\theta_t\omega)}v}dx\right)
\leq2\int_{\mathbb{R}^n}\left(-\beta|v|^2+\psi_1(s,x)\right)dx.
\end{split}
\end{align}
Therefore we have

$$\frac{d}{dt}||v(t,\omega,v_0)||^2+\frac{5}{4}\gamma||v||^2+
\frac{3}{4}\gamma||v||^2+2\beta\int_{\mathbb{R}^n}|v|^2dx\leq
2\int_{\mathbb{R}^n}\psi_1(t,x)dx.$$
By multiplying both sides by \( e^{\frac{5}{4} \gamma t} \) and then integrating the inequality over the interval \((-t, \sigma_0)\), we derive
\begin{align}
\begin{split}
||v(\sigma_0,-t,\omega,v_{-t})||^2&+\left(\frac{3}{4}\gamma+2\beta\right)\int_{-t}^{\sigma_0} e^{\frac{5}{4}
\gamma(s-\sigma_0)}\|v(s,-t,\omega,v_{-t})\|^2ds\\
&\leq e^{\frac{5}{4}
\gamma(-t-\sigma_0)}||v_{-t}||^2+2\int_{-t}^{\sigma_0} e^{\frac{5}{4}
\gamma(s-\sigma_0)}\|\psi_1(s,\cdot)\|_{L^1_{x}}ds.
\end{split}
\end{align}
Since $D \in \mathcal{D}$  is tempered, then there exists a constant $T_1>0$ such that for all $t\geq T_1$,
\begin{equation}
e^{\frac{5}{4}\gamma(-t-\sigma_0)}||v_{-t}||^2\leq 1.
\end{equation}
Then the desired estimates follow immediately.
\end{proof}
\end{lemma}
We now derive uniform estimates of solutions in $H^\alpha(\mathbb{R}^n).$
\begin{lemma}\label{lem2}
  Let $u_0$, $\alpha$, $\sigma$ and $f$  be as in Theorem \ref{theorem1}. Then for every $
\omega\in \Omega$ and $D=\{D(\omega):
\omega\in \Omega\}\in \mathcal{D}$, there exists a positive constant $T=T(\omega,D,\alpha)$
 such that for all $t\geq T$ and $\sigma_0\in[-1,0]$, the solution $v$ of system  (\ref{lin5}) with~$v_{-t}\in D(-t,\theta_{-t}\omega)$ satisfies
 \begin{equation*}
 ||v(\sigma_0,-t,\omega,v_{-t})||^2_{H^\alpha}
\leq C(n,\alpha, \sigma_0,\gamma, \psi_4, \psi_5).
\end{equation*}
\begin{proof}
Multiplying (\ref{lin5}) by $(-\Delta)^{\alpha}v$, we obtain
 $$\frac{d}{dt}||(-\Delta)^{\alpha/2}v||^2+2\gamma||(-\Delta)^{\alpha/2}v||^2=2\mathbf{Im}e^{-iz} (f,(-\Delta)^{\alpha}v).$$
Using the definition of fractional Laplacian operator $(-\Delta)^{\alpha}$ and  Assumption($\bf{A_{f}}$), we have
\begin{align}
\begin{split}
 2\mathbf{Im}e^{-iz} \left(f,(-\Delta)^{\alpha}v\right)&=2C(n,\alpha)\mathbf{Im}e^{-iz}\int_{\mathbb{R}^n}\int_{\mathbb{R}^n}\frac
 {(f(t,x,e^{iz}v(x))-f(t,y,e^{iz}v(y)))\overline{(v(x)-v(y))}}{|x-y|^{n+2\alpha}}dxdy\\
 &=2C(n,\alpha)\mathbf{Im}e^{-iz}\int_{\mathbb{R}^n}\int_{\mathbb{R}^n}\frac
 {(f(t,x,e^{iz}v(x))-f(t,y,e^{iz}v(x)))\overline{(v(x)-v(y))}}{|x-y|^{n+2\alpha}}dxdy\\
 &+2C(n,\alpha)\mathbf{Im}e^{-iz}\int_{\mathbb{R}^n}\int_{\mathbb{R}^n}\frac
 {(f(t,y,e^{iz}v(x))-f(t,y,e^{iz}v(y)))\overline{(v(x)-v(y))}}{|x-y|^{n+2\alpha}}dxdy\\
 &\leq C(n,\alpha)\left\{\int_{\mathbb{R}^n}\int_{\mathbb{R}^n}\frac
 {|\psi_5(x)-\psi_5(y)||v(x)-v(y)|}{|x-y|^{n+2\alpha}}dxdy+
 \int_{\mathbb{R}^n}\int_{\mathbb{R}^n}\psi_4(x)\frac
 {|v(x)-v(y)|^2}{|x-y|^{n+2\alpha}}dxdy\right\}\\
 &\leq C\left(n,\alpha, \psi_4, \psi_5\right)\|(-\Delta)^{\alpha/2}v\|^2+C\left(n,\alpha, \psi_4, \psi_5\right).
 \end{split}
\end{align}

Thus we have
$$\frac{d}{dt}||(-\Delta)^{\alpha/2}v||^2+2\gamma||(-\Delta)^{\alpha/2}v||^2\leq
C(n,\alpha, \psi_4, \psi_5)||(-\Delta)^{\alpha/2}v||^2+C(n,\alpha, \psi_4, \psi_5).$$
Let $\sigma_0\in(-1,0)$ and $r\in(-2,-1).$ Multiplying by $e^{\frac{5}{4}\gamma t}$
and integrating over $(r,\sigma_0)$, we infer that
\begin{align}
\begin{split}
||(-\Delta)^{\alpha/2}v(\sigma_0,-t,\omega,v_{-t})||^2&\leq e^{\frac{5}{4}\gamma
(r-\sigma_0)}||(-\Delta)^{\alpha/2}v(r,-t,\omega,v_{-t})||^2\\
&+C(n,\alpha, \psi_4, \psi_5)\int_r^{\sigma_0}
e^{\frac{5}{4}\gamma(\zeta-\sigma_0)}||(-\Delta)^{\alpha/2}v(\zeta,-t,\omega,v_{-t})||^2d\zeta\\
&+C(n,\alpha, \psi_4, \psi_5)\int_r^{\sigma_0} e^{\frac{5}{4}\gamma(\zeta-\sigma_0)}d\zeta.
\end{split}
\end{align}
Integrating with respect to $r$ on $(-2,-1)$, we obtain
\begin{align}
\begin{split}
||(-\Delta)^{\alpha/2}v(\sigma_0,-t,\omega,v_{-t})||^2&\leq \int_{-2}^{-1}e^{\frac{5}{4}\gamma
(r-\sigma_0)}\|(-\Delta)^{\alpha/2}v(r,-t,\omega,v_{-t})\|^2dr\\
&+C(n,\alpha, \psi_4, \psi_5)\int_{-2}^{-1}\int_r^{\sigma_0}
e^{\frac{5}{4}\gamma(\zeta-\sigma_0)}||(-\Delta)^{\alpha/2}v(\zeta,-t,\omega,v_{-t})||^2d\zeta dr\\
&+C(n,\alpha, \psi_4, \psi_5)\int_{-2}^{-1}\int_r^{\sigma_0} e^{\frac{5}{4}\gamma(\zeta-\sigma_0)}d\zeta dr\\
&\leq \int_{-2}^{-1}e^{\frac{5}{4}\gamma
(r-\sigma_0)}||(-\Delta)^{\alpha/2}v(r,-t,\omega,v_{-t})||^2dr\\
&+C(n,\alpha, \psi_4, \psi_5)\int_{-2}^{\sigma_0}
e^{\frac{5}{4}\gamma(\zeta-\sigma_0)}||(-\Delta)^{\alpha/2}v(\zeta,-t,\omega,v_{-t})||^2d\zeta\\
&+C(n,\alpha, \psi_4, \psi_5)\int_{-2}^{\sigma_0}e^{\frac{5}{4}\gamma(\zeta-\sigma_0)}d\zeta.
\end{split}
\end{align}
Since $\sigma_0\in(-1,0)$, we get
\begin{align}
\begin{split}
\|(-\Delta)^{\alpha/2}v(\sigma_0,-t,\omega,v_{-t})\|^2&\leq \left(1+C(n,\alpha, \psi_4, \psi_5)\right)
\int_{-2}^{\sigma_0} e^{\frac{5}{4}\gamma(\zeta-\sigma_0)}||(-\Delta)^{\alpha/2}v(\zeta,-t,\omega,v_{-t})||^2d\zeta \\&+C(n,\alpha, \sigma_0,\gamma, \psi_4, \psi_5).
\end{split}
\end{align}
Using Gr\"onwall inequality, we have
$$||(-\Delta)^{\alpha/2}v(\sigma_0,-t,\omega,v_{-t})||^2\leq C(n,\alpha, \sigma_0,\gamma, \psi_4, \psi_5).$$
\end{proof}
\end{lemma}

\begin{lemma}\label{lemma333}
 Let $u_0$, $\alpha$, $\sigma$ and $f$  be as in Theorem \ref{theorem1}. Then for every $
\omega\in \Omega$ and $D=\{D(\tau,\omega):
\omega\in \Omega\}\in \mathcal{D}$, there exists a positive constant $T=T(\omega,D,\alpha)\geq2$
 such that for all $t\geq T$ and $\sigma_0\in[-1,0]$, the solution $v$ of the system  (\ref{lin5}) with $v_{-t}\in D(-t,\theta_{-t}\omega)$ satisfies
 \begin{align}\label{lin6}
\begin{split}
\int_{-1}^0\|v(s,-t,\omega,v_{-t})\|_{L^{2}(\mathbb{R})}^{2}ds
&+\int_{-1}^0\left\|\frac{\partial v}{\partial s}(s,-t,\omega,v_{-t})\right\|^2ds
\leq C+C \int_{-\infty}^0
 e^{\frac{5}{4}\gamma(s+1)}||\psi_1(s,\cdot)||_{L^1_x}ds.
 \end{split}
\end{align}
 \begin{proof}
  Multiplying by $v$ and then integrating on $\mathbb{R}^n$,  we get
\begin{align*}
\begin{split}
  \frac{1}{2}\frac{d}{dt}\|v|\|^{2}_{L^2}+\gamma\|v\|^{2}_{L^2}&=\mathbf{Im}e^{-iz}\left(f(t,x,e^{iz}v),v\right)
  \leq-\beta \|v\|_{L^{2}}^{2}+ \int_{\mathbb{R}^n}\psi_1(t,x)dx
  \\&\leq-\beta ||v||_{L^{2}}^{2}+\frac{3}{8}\gamma\|v\|_{L^2}^2+C\int_{\mathbb{R}^n}|\psi_1(t,x)|dx.
\end{split}
\end{align*}
Then
\begin{equation}
\label{lin3}
\frac{d}{dt}||v||_{L^2}^2+\frac{5}{4}\gamma \|v\|_{L^2}^2+2\beta ||v||_{L^{2}}^{2}\leq C(\psi_1).
\end{equation}

Now, multiplying (\ref{lin3}) by   $e^{\frac{5}{4}\gamma t}$
and integrating on $(-1,0)$, we can get
\begin{equation}
\label{res}
2\beta\int_{-1}^0e^{\frac{5}{4}\gamma s}||v(s,-t,\omega,v_{-t})||_{L^{2}(\mathbb{R}^n)}^{2}ds
\leq e^{-\frac{5}{4}\gamma }||v(-1,-t,\omega,v_{-t})||_{L^2(\mathbb{R}^n)}^{2}+C(\psi_1, \gamma).
\end{equation}
 Using the inequality \eqref{res} and Lemma \ref{linli1} with $\sigma_0=-1$, we know that there exists a constant $C$ such that for all $t\geq T_1$,
\begin{equation}
\int_{-1}^0||v(s,-t,\omega,v_{-t})||_{L^{2}(\mathbb{R}^n)}^{2}ds
\leq C+C \int_{-\infty}^0
 e^{\frac{5}{4}\gamma s}||\psi_1(s,\cdot)||_{L^1}ds.
\end{equation}
 Next, we derive a uniform estimate on  $\frac{\partial v}{\partial t}$. Multiplying (\ref{lin5}) by $\frac{\partial v}{\partial t}$ and using  Assumption($\bf{A_{f}}$), we have
 \begin{align}
\begin{split}
\left\|\frac{\partial v}{\partial t}\right\|^2=-\gamma\left(v,\frac{\partial v}{\partial t}\right)+\mathbf{Im}\left(fe^{-iz},\frac{\partial v}{\partial t}\right)&\leq\frac{3}{8}\left\|\frac{\partial v}{\partial t}\right\|^2_{}{}+2\gamma\|v\|^2
+4\|\psi_2(t,\cdot)\|^2_{L^\infty_{x}}\|v\|_{L^{2}}^{2}+4\|\psi_3(t, \cdot)\|^2\\
&\leq \frac{16}{5}\gamma\|v\|^2
+\frac{32}{5}\|\psi_2(t, \cdot)\|^2_{L^\infty_x}\|v\|_{L^{2}_x}^{2}+\frac{32}{5}\|\psi_3(t, \cdot)\|^2_{L^{2}}.\\
\end{split}
\end{align}
Integrating on $(-1,0)$ and using  Lemma \ref{linli1}, we get
 \begin{align}
\begin{split}
\int_{-1}^0\left\|\frac{\partial v}{\partial s}(s,-t,\omega,v_{-t})\right\|^2ds&\leq \frac{16}{5}\int_{-1}^0\gamma\|v\|^2ds
+\frac{32}{5}\int_{-1}^0\|\psi_2(s,\cdot)\|^2_{L^\infty_x}\|v\|_{L^{2}_x}^{2}ds+
\frac{32}{5}\int_{-1}^0\|\psi_3(s, \cdot)\|^2ds\\
&\leq C+C \int_{-\infty}^0
 e^{\frac{5}{4}\gamma(s+1)}||\psi_1||_{L^1}ds.
\end{split}
\end{align}
The proof is completed.
 \end{proof}
\end{lemma}
We now derive the uniform estimates on the tails of solutions. To that end, we
introduce a smooth function $\rho(s)$ defined for $0\leq s<\infty$ such that
$0\leq \rho(s)\leq 1$ for $s\geq 0$ and
\begin{align}
\rho(s)=
\begin{cases}
0, ~\text{if} ~~~0\leq \rho(s)\leq \frac{1}{2},\\
1, ~\text{if} ~~~ \rho(s)\geq 1.
\end{cases}
\end{align}
Note that a positive constant $c_0$ exists such that $\left|\rho'(s)\right|\leq c_0$ for all $s\geq 0$. Moreover, the cut-off function $\rho$ has the following properties, which comes from \cite[Lemma 3.4]{AG18}.
\begin{lemma}
\label{lemma4}
Let $\rho$ be the smooth function. Then
for every $y\in \mathbb{R}^n$ and $k\in \mathbb{N},$ we have
\begin{equation}
\int_{\mathbb{R}^n}\frac{\left|\rho\left(\frac{|x|}{k}\right)-\rho\left(\frac{|y|}{k}\right)\right|^2}
{\left|x-y\right|^{n+2\alpha}}dx\leq \frac{L_1}{k^{2\alpha}},
\end{equation}
where $L_1$ is a positive constant independent of $k\in \mathbb{N}$ and
$y\in \mathbb{R}^n.$
For every $x\in \mathbb{R}^n$ and $k\in \mathbb{N},$ we have
$$\left|(-\Delta)^\alpha\rho_k(x)\right|\leq\frac{L_2}{k^{2\alpha}},$$
where $\rho_k(\cdot)=\rho\left(\frac{|\cdot|}{k}\right)$.
\end{lemma}
The uniform estimate on the tails of solutions in $L^2(\mathbb{R}^n)$
are given below.
\begin{lemma}\label{lem5}
Let $u_0$, $\alpha$, $\sigma$ and $f$ be as in Theorem \ref{theorem1}. Then for every $\varepsilon>0,
\omega\in \Omega$, and for any set  $D=\{D(\omega):
\omega\in \Omega\}\in \mathcal{D}$, there exist two positive constants  $T=T(\omega,D,\varepsilon,\alpha)\geq 1 $ and $
K=K(\omega,\varepsilon)\geq 1$
 such that for all $t\geq T$ and $k\geq K$, the solution $v$ of the system  (\ref{lin5}) with $v_{-t}\in D(-t,\theta_{-t}\omega)$ satisfies
 $$\int_{|x|>k}|v(0,-t,\omega,v_{-t})(x)|^2<\varepsilon.$$
 \begin{proof}
   Multiplying (\ref{lin5}) by $\rho\left(\frac{|x|}{k}\right)v$ and using Assumption($\bf{A_{f}}$), we obtain
  \begin{align}
\begin{split}
   \frac{1}{2}\frac{d}{dt}\int_{\mathbb{R}^n}\rho\left(\frac{|x|}{k}\right)|v|^2dx&+\gamma
\int_{\mathbb{R}^n}\rho\left(\frac{|x|}{k}\right)|v|^2dx=\mathbf{Im}e^{-iz}\int_{\mathbb{R}^n} f\rho\left(\frac{|x|}{k}\right)\overline{v}
   dx\\
   &\leq-\beta\int_{\mathbb{R}^n} |v|^2\rho\left(\frac{|x|}{k}\right)dx+\int_{\mathbb{R}^n} \psi_1(t,x)\rho\left(\frac{|x|}{k}\right)dx\\
   &\leq\int_{\mathbb{R}^n} \psi_1(t,x)\rho\left(\frac{|x|}{k}\right)dx.
   \end{split}
\end{align}
Thus
\begin{equation}
\label{rho}
\frac{d}{dt}\int_{\mathbb{R}^n}\rho\left(\frac{|x|}{k}\right)|v|^2dx+2\gamma
\int_{\mathbb{R}^n}\rho\left(\frac{|x|}{k}\right)|v|^2dx\leq\int_{|x|\geq\frac{k}{2}} \psi_1(t,x)dx.
\end{equation}
By multiplying the equality \eqref{rho} by $e^{\frac{5}{4}\gamma t}$ and integrating over the interval $(-t,0)$, we know that
  there exists a constant $K_1\geq1$ such that for all $k\geq K_1$,
\begin{equation}
\int_{\mathbb{R}^n}\rho\left(\frac{|x|}{k}\right)|v(0,-t,\omega,v_{-t})|^2dx-e^{-\frac{5}{4}\gamma t}
\int_{\mathbb{R}^n}\rho\left(\frac{|x|}{k}\right)|v_{-t}(x)|^2dx\leq\int_{-t}^0\int_{|x|\geq\frac{k}{2}} e^{\frac{5}{4}\gamma s}\psi_1(s,x)dxds.
\end{equation}
Since $v_{-t}\in D(-t,\theta_{-t}\omega)$ and $D$ is tempered,  there exists
$T_1>0$ such that for all $t\geq T_1$
\begin{equation}
e^{-\frac{5}{4}\gamma t}
\int_{\mathbb{R}^n}\rho\left(\frac{|x|}{k}\right)|v_{-t}(x)|^2dx\leq \frac{\varepsilon}{2}.
\end{equation}
Observe that
$$\int_{-\infty}^0\int_{\mathbb{R}^n} e^{\frac{5}{4}\gamma s}\psi_1(s,x)dxds<\infty.$$
 Hence there exists a positive constant $K_2\geq K_1$ such that for all $k\geq K_2$,
 \begin{equation}
 \int_{-t}^0\int_{|x|\geq\frac{k}{2}} e^{\frac{5}{4}\gamma s}\psi_1(s,x)
 dxds\leq\frac{\varepsilon}{2}.
 \end{equation}
 This yields  $$\int_{|x|>k}|v(0,-t,\omega,v_{-t})(x)|^2<\varepsilon.$$
  \end{proof}
\end{lemma}
Next, we derive the uniform estimates on the tails of solutions in
$H^\alpha(\mathbb{R}^n).$
\begin{lemma}\label{lem6}
Let $u_0$, $\alpha$, $\sigma$ and $f$ be as in Theorem \ref{theorem1}. Then for every $\varepsilon>0,
\omega\in \Omega$ and $D=\{D(\omega):\tau\in\mathbb{R},
\omega\in \Omega\}\in \mathcal{D}$, there exist constants $T=T(\omega,D,\varepsilon,\alpha)\geq 1 $ and $K=K(\omega,\varepsilon)\geq 1$
 such that for all $t\geq T$ and $k\geq K$, the solution $v$ of system  (\ref{lin5}) with the initial value  $v_{-t}\in D(-t,\theta_{-t}\omega)$ satisfies
 the following conditions:
$$\int_{\mathbb{R}^n}\int_{|x|\geq k}\frac{|v(0,-t,\omega,v_{-t})(x)-v(0,-t,\omega,v_{-t})(y)|^2}{|x-y|^{n+2\alpha}}dxdy\leq\varepsilon.$$
\begin{proof}
Multiplying (\ref{lin5}) by $\rho\left(\frac{|x|}{k}\right)(-\Delta)^{\alpha}v$, we obtain
\begin{equation}
\label{eq}
\left(\frac{\partial v}{\partial t},\rho\left(\frac{|x|}{k}\right)(-\Delta)^{\alpha}v\right)+\gamma\left(v,\rho\left(\frac{|x|}{k}\right)(-\Delta)^{\alpha}v\right)
=\mathbf{Im}\left(fe^{-iz},\rho\left(\frac{|x|}{k}\right)(-\Delta)^{\alpha}v\right).
\end{equation}
On the one hand, by the definition of $(-\Delta)^{\alpha}$, H\"older inequality and Lemma \ref{lemma4}, we have
\begin{align}
\label{27}
\begin{split}
-\int_{\mathbb{R}^n} \frac{\partial v}{\partial t}\rho\left(\frac{|x|}{k}\right)(-\Delta)^{\alpha}\overline{v}dx
&=-\frac{1}{2}C(n,\alpha)\int_{\mathbb{R}^n}\int_{\mathbb{R}^n}
\frac{\left(\frac{\partial v(t,x)}{\partial t}\rho\left(\frac{|x|}{k}\right)-\frac{\partial v(t,y)}{\partial t}\rho\left(\frac{|y|}{k}\right)\right)\overline{(v(t,x)-v(t,y)})}
{|x-y|^{n+2\alpha}}dxdy\\
&=-\frac{1}{2}C(n,\alpha)\int_{\mathbb{R}^n}\int_{\mathbb{R}^n}
\frac{\rho\left(\frac{|x|}{k}\right)\left(\frac{\partial v(t,x)}{\partial t}-\frac{\partial v(t,y)}{\partial t}\right)\overline{(v(t,x)-v(t,y)})}
{|x-y|^{n+2\alpha}}dxdy\\
&\quad-\frac{1}{2}C(n,\alpha)\int_{\mathbb{R}^n}\int_{\mathbb{R}^n}
\frac{\frac{\partial v(t,y)}{\partial t}\left(\rho\left(\frac{|x|}{k}\right)
-\rho\left(\frac{|y|}{k}\right)\right)\overline{(v(t,x)-v(t,y)})}
{|x-y|^{n+2\alpha}}dxdy\\
&=-\frac{1}{4}C(n,\alpha)\frac{d}{dt}\int_{\mathbb{R}^n}\int_{\mathbb{R}^n}
\frac{\rho\left(\frac{|x|}{k}\right)\left(v(t,x)-v(t,y)\right)^2}{|x-y|^{n+2\alpha}}dxdy\\
&\quad+\frac{1}{2}C(n,\alpha)\left\|\frac{\partial v(t,\cdot)}{\partial t}\right\|_{L^2_{\cdot}}\left(\int_{\mathbb{R}^n}\left(\int_{\mathbb{R}^n}
\frac{\left(\rho\left(\frac{|x|}{k}\right)
-\rho\left(\frac{|y|}{k}\right)\right)\overline{(v(t,x)-v(t,y))}}{|x-y|^{n+2\alpha}}dx
\right)^2dy
\right)^{\frac{1}{2}}\\
&\leq-\frac{1}{4}C(n,\alpha)\frac{d}{dt}\int_{\mathbb{R}^n}\int_{\mathbb{R}^n}
\frac{\rho\left(\frac{|x|}{k}\right)\left(v(t,x)-v(t,y)\right)^2}{|x-y|^{n+2\alpha}}dxdy
+C(n, \alpha)k^{-\alpha}\left\|\frac{\partial v}{\partial t}\right\|\cdot\|v\|_{\dot{H}^\alpha},
\end{split}
\end{align}
and
\begin{align}
\label{28}
\begin{split}
-\gamma\int_{\mathbb{R}^n} v\rho\left(\frac{|x|}{k}\right)(-\Delta)^{\alpha}\overline{v}dx
&\leq
-\frac{\gamma}{2}C(n,\alpha)\int_{\mathbb{R}^n}\int_{\mathbb{R}^n}
\frac{\rho\left(\frac{|x|}{k}\right)\left(v(t,x)-v(t,y)\right)^2}
{|x-y|^{n+2\alpha}}dxdy\\
&\quad-\frac{\gamma}{2}C(n,\alpha)\int_{\mathbb{R}^n}\int_{\mathbb{R}^n}
\frac{v(t,y)\left(\rho\left(\frac{|x|}{k}\right)
-\rho\left(\frac{|y|}{k}\right)\right)\overline{(v(t,x)-v(t,y)})}
{|x-y|^{n+2\alpha}}dxdy\\
&\leq-\frac{\gamma}{2}C(n,\alpha)\int_{\mathbb{R}^n}\int_{\mathbb{R}^n}
\frac{\rho\left(\frac{|x|}{k}\right)\left(v(t,x)-v(t,y)\right)^2}
{|x-y|^{n+2\alpha}}dxdy+C(n,\alpha)k^{-\alpha}\|v\|^2_{\dot{H}^\alpha}.
\end{split}
\end{align}
On the other hand, by the definition of $(-\Delta)^{\alpha}$, we have
\begin{align}\label{llin1}
\begin{split}
\mathbf{Im}\left(fe^{-iz},\rho\left(\frac{|x|}{k}\right)(-\Delta)^{\alpha}v\right)&=C(n,\alpha)
\mathbf{Im} e^{-iz}\int_{\mathbb{R}^n}\int_{\mathbb{R}^n}
\frac{\left(f(t,x,e^{iz}v(t,x))\rho\left(\frac{|x|}{k}\right)-f(t,y,e^{iz}v(t,y)\rho\left(\frac{|y|}{k}\right)\right)\overline{(v(t,x)-v(t,y)})}
{|x-y|^{n+2\alpha}}dxdy\\
&=C(n,\alpha)\mathbf{Im} e^{-iz}\int_{\mathbb{R}^n}\int_{\mathbb{R}^n}
\frac{f\left(t,y,e^{iz}v(t,y)\right)\left(\rho\left(\frac{|y|}{k}\right)-\rho\left(\frac{|x|}{k}\right)\right)\overline{(v(t,y)-v(t,x)})}
{|x-y|^{n+2\alpha}}dxdy\\
&\quad+C(n,\alpha)\mathbf{Im} e^{-iz}\int_{\mathbb{R}^n}\int_{\mathbb{R}^n}
\frac{\rho\left(\frac{|x|}{k}\right)\left(f(t,y,e^{iz}v(t,y))-f(t,x,e^{iz}v(t,y)\right)\overline{(v(t,y)-v(t,x)})}
{|x-y|^{n+2\alpha}}dxdy\\
&\quad +C(n,\alpha)\mathbf{Im} e^{-iz}\int_{\mathbb{R}^n}\int_{\mathbb{R}^n}
\frac{\rho\left(\frac{|x|}{k}\right)\left(f(t,x,e^{iz}v(t,y))-f(t,x,e^{iz}v(t,x)\right)\overline{(v(t,y)-v(t,x)})}
{|x-y|^{n+2\alpha}}dxdy \\
&:=\mathcal{K}_1+\mathcal{K}_2+\mathcal{K}_3.
\end{split}
\end{align}

For the term $\mathcal{K}_1$, using H\"older inequality and
Assumption($\bf{A_{f}}$), we have
\begin{align}
\label{K1}
\begin{split}
\mathcal{K}_1&:=C(n, \alpha)\mathbf{Im} e^{-iz}\int_{\mathbb{R}^n}\int_{\mathbb{R}^n}
\frac{f\left(t,y,e^{iz}v(t,y)\right)\left(\rho\left(\frac{|y|}{k}\right)-\rho\left(\frac{|x|}{k}\right)\right)\overline{(v(t,y)-v(t,x)})}
{|x-y|^{n+2\alpha}}dxdy\\
&\leq C(n, \alpha)
\left\|f\left(t,\cdot,e^{iz}v\right)\right\|_{L^2}
\left(\int_{\mathbb{R}^n}\left(\int_{\mathbb{R}^n}
\frac{\left(\rho\left(\frac{|x|}{k}\right)
-\rho\left(\frac{|y|}{k}\right)\right)|(v(t,x)-v(t,y))|}{|x-y|^{n+2\alpha}}dx
\right)^2dy
\right)^{\frac{1}{2}}\\
&\leq C(n, \alpha) k^{-\alpha}\|f\left(t,\cdot,e^{iz}v\right)\|_{L^2}\cdot\|v(t,\cdot)\|_{\dot{H}^\alpha}\\
&\leq C(n, \alpha) k^{-\alpha}\int_{\mathbb{R}^n}|\psi_2(y)|^2|v(t,y)|^{2}dy+C(n, \alpha) k^{-\alpha}\|\psi_3(t,\cdot)\|^2_{L^2}+
C(n, \alpha) k^{-\alpha}\|v(t,\cdot)\|_{{H}^\alpha}^2.
\end{split}
\end{align}
For the terms $\mathcal{K}_2$ and $\mathcal{K}_3$, using H\"older inequality and
Assumption($\bf{A_{f}}$), we have
\begin{align}
\begin{split}
\mathcal{K}_2&:=C(n, \alpha)\mathbf{Im} e^{-iz}\int_{\mathbb{R}^n}\int_{\mathbb{R}^n}
\frac{\rho\left(\frac{|x|}{k}\right)\left(f(t,y,e^{iz}v(t,y))-f(t,x,e^{iz}v(t,y)\right)\overline{(v(t,y)-v(t,x)})}
{|x-y|^{n+2\alpha}}dxdy\\
&\leq\frac{C(n, \alpha)}{2}\int_{\mathbb{R}^n}\int_{\mathbb{R}^n}
\frac{\rho\left(\frac{|x|}{k}\right)(\psi_5(x)-\psi_5(y))^2}{|x-y|^{n+2\alpha}}dxdy
+\frac{C(n, \alpha)}{2}\int_{\mathbb{R}^n}\int_{\mathbb{R}^n}
\frac{\rho\left(\frac{|x|}{k}\right)\left(v(t,x))-v(t,y)\right)^2}{|x-y|^{n+2\alpha}}dxdy,
\end{split}
\end{align}
and
\begin{align}
\label{K3}
\begin{split}
\mathcal{K}_3&:=C(n, \alpha)\mathbf{Im} e^{-iz}\int_{\mathbb{R}^n}\int_{\mathbb{R}^n}
\frac{\rho\left(\frac{|x|}{k}\right)\left(f(t,x,e^{iz}v(t,y))-f(t,x,e^{iz}v(t,x)\right)\overline{(v(t,y)-v(t,x)})}
{|x-y|^{n+2\alpha}}dxdy\\
&\leq C(n, \alpha)\int_{\mathbb{R}^n}\int_{\mathbb{R}^n}
\frac{\rho\left(\frac{|x|}{k}\right)\psi_4(t,x)(v(t,x)-v(t,y))^2}
{|x-y|^{n+2\alpha}}dxdy\\
&\leq C(n,\alpha,\psi_4)
\int_{\mathbb{R}^n}\int_{\mathbb{R}^n}
\frac{\rho\left(\frac{|x|}{k}\right)(v(t,x)-v(t,y))^2}
{|x-y|^{n+2\alpha}}dxdy.
\end{split}
\end{align}
Combing the equalities \eqref{K1}--\eqref{K3}, we obtain the following result:
\begin{align}
\label{33}
\begin{split}
\mathbf{Im}&\left(fe^{-iz},\rho\left(\frac{|x|}{k}\right)(-\Delta)^{\alpha}v\right)\leq
C(n, \alpha)k^{-\alpha}\int_{\mathbb{R}^n}|\psi_2(x)|^2|v(t,x)|^{2}dx+C(n, \alpha)k^{-\alpha}\|\psi_3(t,\cdot)\|^2_{L^2}+
C(n, \alpha)k^{-\alpha}\|v(t,\cdot)\|_{{H}^\alpha}^2\\
&+\frac{C(n, \alpha)}{2}\int_{\mathbb{R}^n}\int_{\mathbb{R}^n}
\frac{\rho\left(\frac{|x|}{k}\right)(\psi_5(x)-\psi_5(y))^2}{|x-y|^{n+2\alpha}}dxdy
+C(n,\alpha,\psi_4)\int_{\mathbb{R}^n}\int_{\mathbb{R}^n}
\frac{\rho\left(\frac{|x|}{k}\right)\left(v(t,x))-v(t,y)\right)^2}{|x-y|^{n+2\alpha}}dxdy.\\
\end{split}
\end{align}
Thus using the equality \eqref{eq} and the inequalities\eqref{27}, \eqref{28}, \eqref{33}, along with   H\"older inequality,  we obtain
\begin{align}
\begin{split}
\frac{d}{dt}\int_{\mathbb{R}^n}\int_{\mathbb{R}^n}
\frac{\rho\left(\frac{|x|}{k}\right)\left(v(t,x)-v(t,y)\right)^2}{|x-y|^{n+2\alpha}}dxdy
& \leq C(n,\alpha)k^{-\alpha}\left\|\frac{\partial v}{\partial t}\right\|^2
+C(n,\alpha)k^{-\alpha}\int_{\mathbb{R}^n}\psi_2^2(x)|v(t,x)|^{2}dx
+C(n,\alpha)k^{-\alpha}\|\psi_3(t, \cdot)\|^2\\
&+C(n,\alpha)k^{-\alpha}\|v(t, \cdot)\|_{{H}^\alpha}^2
+C(n,\alpha)\int_{\mathbb{R}^n}\int_{\mathbb{R}^n}
\frac{\rho\left(\frac{|x|}{k}\right)(\psi_5(x)-\psi_5(y))^2}{|x-y|^{n+2\alpha}}dxdy\\
&+C(n,\alpha, \psi_4)\int_{\mathbb{R}^n}\int_{\mathbb{R}^n}
\frac{\rho\left(\frac{|x|}{k}\right)\left(v(t,x))-v(t,y)\right)^2}{|x-y|^{n+2\alpha}}dxdy.\\
\end{split}
\end{align}
Let $r\in(-1,0)$. First integrating with respect to $t$ on $(r,0)$, we obtain
\begin{align}\label{eqq11}
\begin{split}
&\int_{\mathbb{R}^n}\int_{\mathbb{R}^n}
\frac{\rho\left(\frac{|x|}{k}\right)\left|v(0,-t,\omega,v_{-t})(x)-v(0,-t,\omega,v_{-t})(y)
\right|^2}{|x-y|^{n+2\alpha}}dxdy\\
&\leq\int_{\mathbb{R}^n}\int_{\mathbb{R}^n}
\frac{\rho\left(\frac{|x|}{k}\right)\left|v(r,-t,\omega,v_{-t})(x)-v(r,-t,\omega,v_{-t})(y)
\right|^2}{|x-y|^{n+2\alpha}}dxdy\\
&+C(n,\alpha)k^{-\alpha}\int_{-1}^0\|\frac{\partial v}{\partial\zeta}(\zeta,-t,\omega,v_{-t})\|^2d\zeta \\&+C(n,\alpha)k^{-\alpha}\int_{-1}
^0\|\psi_2(\cdot)\|_{L^\infty(\mathbb
{R}^n)}^2\int_{\mathbb{R}^n}|v(\zeta, x)|^{2}dxd\zeta\\
&+C(n,\alpha)k^{-\alpha}\int_{-1}^0\|\psi_3(\zeta, \cdot)\|^2d\zeta+C(n,\alpha)k^{-\alpha}\int_{-1}^0\|v(\zeta, \cdot)\|_{{H}^\alpha}^2d\zeta\\
&+C(n,\alpha)\int_{\mathbb{R}^n}\int_{\mathbb{R}^n}
\frac{\rho\left(\frac{|x|}{k}\right)(\psi_5(x)-\psi_5(y))^2}{|x-y|^{n+2\alpha}}dxdy\\
&+C(n,\alpha, \psi_4)\int_{-1}^0\int_{\mathbb{R}^n}\int_{\mathbb{R}^n}
\frac{\rho\left(\frac{|x|}{k}\right)\left(v(\zeta,x)-v(\zeta,y)\right)^2}{|x-y|^{n+2\alpha}}dxdyd\zeta
\end{split}
\end{align}

For the first term on the right side of inequality \eqref{eqq11}, we integrate the variable $r$ over the interval  $(-1,0)$. Then according to Lemma \ref{lem2}, there exist constants $T_1(\omega,\varepsilon)\geq1$
and $K_1(\omega,\varepsilon)\geq1$ such that for all $t\geq T_1$ and $k\geq K_1,$
\begin{equation}\label{eqq12}
 C(n,\alpha, \psi_4)\int_{-1}^0\int_{\mathbb{R}^n}\int_{\mathbb{R}^n}
\frac{\rho(\frac{|x|}{k})\left|v(r,-t,\omega,v_{-t})(x)-v(r,-t,\omega,v_{-t})(y)
\right|^2}{|x-y|^{n+2\alpha}}dxdydr\leq \frac{\varepsilon}{3}.
\end{equation}
By integrating the inequality (\ref{eqq11}) with respect to $r$ over the interval $(-1,0)$, and applying the inequalities (\ref{eqq12}), as well as Lemma \ref{lem2}, and Lemma \ref{lemma333}, we conclude that there exist constants $T_2(\varepsilon) \geq T_1$ and $K_2(\varepsilon) \geq K_1$ such that for all $t \geq T_2$ and $k \geq K_2$,
\begin{align}\label{eqq33}
\begin{split}
 &\int_{\mathbb{R}^n}\int_{\mathbb{R}^n}
\frac{\rho\left(\frac{|x|}{k}\right)\left|v(0,-t,\omega,v_{-t})(x)-v(0,-t,\omega,v_{-t})(y)
\right|^2}{|x-y|^{n+2\alpha}}dxdy\\
&\leq \frac{\varepsilon}{3}+ck^{-s}+ C(n,\alpha)\int_{\mathbb{R}^n}\int_{|x|\geq\frac{k}{2}}\frac{
(\psi_5(x)-\psi_5(y))^2}{|x-y|^{n+2\alpha}}dxdy.
\end{split}
\end{align}
By Assumption ($\bf{A_{f}}$), we know that $\psi_5\in H^\alpha(\mathbb{R}^n)$. Thus there exists a constant $K_3\geq K_2$ such that for all $k\geq K_3$
$$\frac{\varepsilon}{3}+ck^{-s}+C(n,\alpha)\int_{\mathbb{R}^n}\int_{|x|\geq\frac{k}{2}}\frac{
(\psi_5(x)-\psi_5(y))^2}{|x-y|^{n+2\alpha}}dxdy\leq\varepsilon,$$
which along with $(\ref{eqq33})$ concludes the proof.
\end{proof}
\end{lemma}
\begin{lemma}\label{lem7}
Let $u_0$, $\sigma$ and $f$  be  as in Theorem \ref{theorem1}. Then for $
\omega\in \Omega$ and $D=\{D(\omega):
\omega\in \Omega\}\in \mathcal{D}$, there exists a constant $T=T(\omega,D,\varepsilon,\alpha)\geq 1,$
 such that for all $t\geq T$, the solution $v$ of system  (\ref{lin5}) with $v_{-t}\in D(-t,\theta_{-t}\omega)$ satisfies the following inequality
 $$\|(-\Delta)^{\frac{\alpha}{2}+\frac{1}{4}}v(\sigma_0,-t,\omega,v_{-t})\|^2\leq C(n,\alpha, \psi_7, \psi_6,\sigma_0,\gamma).$$
\begin{proof}
  Taking the inner product of (\ref{lin5}) with $(-\Delta)^{\alpha+\frac{1}{2}}v$ and considering  the real part, we obtain
\begin{equation}
  \frac{d}{dt}\|(-\Delta)^{\frac{\alpha}{2}+\frac{1}{4}}v\|^2+2\gamma\|(-\Delta)^{\frac{\alpha}{2}+\frac{1}{4}}v\|^2
=2\mathbf{Im}\left(f(t,x,e^{i z(\theta_t\omega)}v)e^{-i z(\theta_t\omega)},(-\Delta)^{\alpha+\frac{1}{2}}v\right).
\end{equation}
  Let $\psi(t,x)=e^{i z(\theta_t\omega)}v(t,x)$. By integrating by parts and using  H\"older inequality, Young's inequality and Assumption ($\bf{A_{f}}$), we obtain
\begin{align}
\begin{split}
&\quad2\mathbf{Im}\left(f(t,x,e^{i z(\theta_t\omega)}v)e^{-i z(\theta_t\omega)},(-\Delta)^{\alpha+\frac{1}{2}}v\right)=
2\mathbf{Im} e^{-i z(\theta_t\omega)} \left(f(t,x,\psi),(-\Delta)^{\alpha+\frac{1}{2}}v\right)\\
&\leq|\left(f_x+f_\psi\nabla\psi,
(-\Delta)^{\alpha}v\right)|\\
&=2C(n,\alpha)\int_{\mathbb{R}^n}\int_{\mathbb{R}^{m}}\frac
 {\left(f_x(t,x,e^{iz}v(t,x))-f_x(t,y,e^{iz}v(t,x))\right)\overline{(v(t,x)-v(t,y))}}{|x-y|^{n+2\alpha}}dxdy\\
&+2C(n,\alpha)\int_{\mathbb{R}^n}\int_{\mathbb{R}^{m}}\frac
 {\left(f_x(t,y,e^{-iz}v(t,x))-f_x(t,y,e^{iz}v(t,y))\right)\overline{(v(t,x)-v(t,y))}}{|x-y|^{n+2\alpha}}dxdy+|\left(f_\psi\nabla\psi,
(-\Delta)^{\alpha}v\right)|\\
 &\leq C(n,\alpha)\left\{\int_{\mathbb{R}^n}\int_{\mathbb{R}^m}\frac
 {|\psi_7(x)-\psi_7(y)||v(t,x)-v(t,y)|}{|x-y|^{n+2\alpha}}dxdy+
 \int_{\mathbb{R}^n}\int_{\mathbb{R}^m}\psi_6(x)\frac
 {|v(t,x)-v(t,y)|^2}{|x-y|^{n+2\alpha}}dxdy\right\}\\
 &+|\left(f_\psi\nabla\psi,
(-\Delta)^{\alpha}v\right)|\\
 &\leq C(n,\alpha, \psi_7, \psi_6)||(-\Delta)^{\frac{\alpha}{2}}v||^2+\|f_\psi\|_{L_x^\infty}\|(-\Delta)^{\frac{\alpha}{2}+\frac{1}{4}}v\|^2\\
&\leq C(n,\alpha, \psi_7, \psi_6)\|(-\Delta)^{\frac{\alpha}{2}}v\|^2+C_1(t)\|(-\Delta)^{\frac{\alpha}{2}+\frac{1}{4}}v\|^2,
\end{split}
\end{align}
where $C_1(t)\in L^{\infty}(\mathbb{R})$.

Thus we have
$$\frac{d}{dt}\|(-\Delta)^{\frac{\alpha}{2}+\frac{1}{4}}v\|^2+2\gamma\|(-\Delta)
^{\frac{\alpha}{2}+\frac{1}{4}}v\|^2\leq C(n,\alpha, \psi_7, \psi_6)\|(-\Delta)^{\frac{\alpha}{2}}v\|^2+C_1(t)\|(-\Delta)^{\frac{\alpha}{2}+\frac{1}{4}}v\|^2.$$
Let $\sigma_0\in(-\frac{1}{2},0)$ and $r\in(-1,-\frac{1}{2}).$ Multiplying by $e^{\frac{5}{4}\gamma t}$ and integrating over $(r,\sigma_0)$, we obtain
\begin{align}
\begin{split}
\|(-\Delta)^{\frac{\alpha}{2}+\frac{1}{4}}v(\sigma_0,-t,\omega,v_{-t})\|^2&\leq
e^{\frac{5}{4}\gamma (r-\sigma_0)}\|(-\Delta)^{\frac{\alpha}{2}+\frac{1}{4}}v(r,-t,\omega,v_{-t})\|^2\\
&\quad+\int_r^{\sigma_0}C_1(r)e^{\frac{5}{4}\gamma (\zeta-\sigma_0)}\|(-\Delta)^{\frac{\alpha}{2}+\frac{1}{4}} v(\zeta,-t,\omega,v_{-t})\|^2d\zeta\\
&\quad+\int_r^{\sigma_0} C(n,\alpha, \psi_7, \psi_6)e^{\frac{5}{4}\gamma (\zeta-\sigma_0)}\|(-\Delta)^\frac{\alpha}{2} v(\zeta,-t,\omega,v_{-t})\|^2d\zeta.
\end{split}
\end{align}
Integrating with respect to $r$ on $(-1,-\frac{1}{2})$, we obtain
\begin{align}
\begin{split}
\frac{1}{2}\|(-\Delta)^{\frac{\alpha}{2}+\frac{1}{4}}v(\sigma_0,-t,\omega,v_{-t})\|^2
&\leq\int_{-1}^{-\frac{1}{2}}e^{\frac{5}{4}\gamma (r-\sigma_0)}\|(-\Delta)^{\frac{\alpha}{2}+\frac{1}{4}}v(r,-t,\omega,v_{-t})\|^2dr\\
&\quad+\int_{-1}^{-\frac{1}{2}}\int_r^{\sigma_0} C_1(r)e^{\frac{5}{4}\gamma (\zeta-\sigma_0)}\|(-\Delta)^{\frac{\alpha}{2}+\frac{1}{4}} v(\zeta,-t,\omega,v_{-t})\|^2d\zeta dr\\
&\quad+\int_{-1}^{-\frac{1}{2}}\int_r^{\sigma_0}C(n,\alpha, \psi_7, \psi_6)e^{\frac{5}{4}\gamma (\zeta-\sigma_0)}\|(-\Delta)^\frac{\alpha}{2} v(\zeta,-t,\omega,v_{-t})\|^2d\zeta dr\\
&\leq\int_{-1}^{\sigma_0} e^{\frac{5}{4}\gamma (r-\sigma_0)}\|(-\Delta)^{\frac{\alpha}{2}+\frac{1}{4}}v(r,-t,\omega,v_{-t})\|^2dr\\
&\quad+\frac{C}{2}\int_{-1}^{\sigma_0} e^{\frac{5}{4}\gamma (\zeta-\sigma_0)}\|(-\Delta)^{\frac{\alpha}{2}+\frac{1}{4}} v(\zeta,-t,\omega,v_{-t})\|^2d\zeta \\
&\quad+\frac{C(n,\alpha, \psi_7, \psi_6)}{2}\int_{-1}^{\sigma_0} e^{\frac{5}{4}\gamma (\zeta-\sigma_0)}\|(-\Delta)^\frac{\alpha}{2} v(\zeta,-t,\omega,v_{-t})\|^2d\zeta.
\end{split}
\end{align}
Using Lemma \ref{lem2} and Gr\"onwall inequality, we obtain
$$\|(-\Delta)^{\frac{\alpha}{2}+\frac{1}{4}}v(\sigma_0,-t,\omega,v_{-t})\|^2\leq C(n,\alpha, \psi_7, \psi_6,\sigma_0,\gamma).$$
\end{proof}
\end{lemma}

\subsection{Existence of random attractor}
In this subsection, we prove the existence of a random attractor for the random
dynamical system generated by (\ref{0lin}).
\begin{lemma}
\label{compact}
Let $u_0$, $\sigma$ and $f$ and be as in Theorem \ref{theorem1}. Then for every $
\omega\in \Omega$ and $D=\{D(\omega):
\omega\in \Omega\}\in \mathcal{D}$, the sequence $\{\Phi(t_n,-t_n,\theta_{-t_n}\omega,u_{0,n})\}_{n=1}^\infty$ has a
convergent subsequence in $H^{\alpha}(\mathbb{R}^n)$ whenever $t_n\rightarrow\infty$ and
$u_{0,n}\in D(-t_n,\theta_{-t_n}\omega)$.
\begin{proof}
Since $\Phi(t_n,-t_n,\theta_{-t_n}\omega,u_{0,n})=u(0,-t_n,\omega,u_{-t_n})
=e^{iz(\omega)}v(0,-t_n,\omega,v_{-t_n}).$
Applying Lemma \ref{linli1} and Lemma \ref{lem2}, we know
\begin{equation}
\{\Phi(t_n,-t_n,\theta_{-t_n}\omega,u_{0,n})\}_{n=1}^\infty \quad \text{is bounded in }  H^\alpha(\mathbb{R}^n).
\end{equation}
Therefore there exists a function $\eta\in{H^\alpha(\mathbb{R}^n)}$ and a subsequence, which we will denote for convenience as  $\Phi(t_n,-t_n,\theta_{-t_n}\omega,u_{0,n}),$ such that
$$\Phi(t_n,-t_n,\theta_{-t_n}\omega,u_{0,n})\rightarrow \eta
\quad \text{weakly in }  H^\alpha(\mathbb{R}^n).$$
Given $\varepsilon>0,$ by  Lemma \ref{lem5} and Lemma \ref{lem6}, there exist two positive constants
$T$ and $R^*=R^*(\varepsilon)$  such that for $t\geq T$
$$\|\Phi(t,-t,\theta_{-t}\omega,v_0(\theta_{-t}\omega))\|^2_{H^\alpha(|x|\geq R^*)}\leq\varepsilon.$$
Since $t_n\rightarrow\infty$, there exists an integer $N$ such that $t_n\geq T$ for all
$n\geq N$. Thus we have
$$ \|\Phi(t_n,-t_n,\theta_{-t_n}\omega,u_{0,n})\|^2_{H^\alpha(|x|\geq R^*)}\leq\varepsilon, $$
and
$$\|\eta\|^2_{H^\alpha(|x|\geq R^*)}\leq\varepsilon.$$
Applying Lemma \ref{linli1} and Lemma \ref{lem7}, there exists a constant $T$ such that for
all $t\geq T$
$$\|\Phi(t,-t,\theta_{-t}\omega,v_0(\theta_{-t}\omega))\|^2_{H^{\alpha+\frac{1}{2}}(\mathbb{R}^n)}
\leq C.$$
Let $N_2$ be sufficiently large such that $t_n\geq T_2$ for $n\geq N_2$. Consequently, for all $n\geq N_2$
$$\|\Phi(t_n,-t_n,\theta_{-t_n}\omega,u_{0,n})\|^2_{H^{\alpha+\frac{1}{2}}(\mathbb{R}^n)}
\leq C.$$
Let $B_{R^*}=\{x\in \mathbb{R}: |x|\leq R^*\}$. Due to the compactness of the embedding $H^{\alpha+\frac{1}{2}}(B_{R^*})\hookrightarrow H^{\alpha}(B_{R^*})$,
we can conclude that, up to a subsequence,  the sequence depends
on $R^*$ and that $\Phi(t_n,-t_n,\theta_{-t_n}\omega,u_{0,n})\rightarrow \eta$ converges strongly in
$H^\alpha(B_{R^*})$. This implies that there exists a constant $N_3\geq N_2$ such that for all
$n\geq N_3$,
$$\|\Phi(t_n,-t_n,\theta_{-t_n}\omega,u_{0,n})-\eta\|^2_{H^{\alpha}(B_{R^*})}\leq\varepsilon.$$
Let $N=max\{N_1,N_3\}$. Then for all $n\geq N$, we have
\begin{equation}
\begin{aligned}
\|\Phi(t_n,-t_n,\theta_{-t_n}\omega,u_{0,n})-\eta\|^2_{H^{\alpha}(\mathbb{R}^n)}
&\leq \|\Phi(t_n,-t_n,\theta_{-t_n}\omega,u_{0,n})-\eta\|^2_{H^{\alpha}(B_{R^*})}\\
&\quad+\|\Phi(t_n,-t_n,\theta_{-t_n}\omega,u_{0,n})-\eta\|^2_{H^{\alpha}(|x|\geq R^*)}\\
&\leq \|\Phi(t_n,-t_n,\theta_{-t_n}\omega,u_{0,n})-\eta\|^2_{H^{\alpha}(B_{R^*})}\\
&\quad+\|\Phi(t_n,-t_n,\theta_{-t_n}\omega,u_{0,n})\|^2_{H^{\alpha}(|x|\geq R^*)}+\|\eta\|^2_{H^{\alpha}(|x|\geq R^*)}\\
&\leq 3\varepsilon,
\end{aligned}
\end{equation}
which implies that $$\Phi(t_n,-t_n,\theta_{-t_n}\omega,u_{0,n})\rightarrow \eta
\quad \text{strongly in }  H^\alpha(\mathbb{R}^n).$$
This completes the proof.
\end{proof}
\end{lemma}
\subsection {Proof of Theorem \ref{theorem2}}
Proof of Theorem  \ref{theorem2}.
By the uniqueness of the solution, we know that $\Phi(t+s, \varrho)=\Phi(t, s+\varrho)\circ \Phi(s, \varrho)$ for all $t, s \in \mathbb{R}^{+}$ and $\varrho \in \mathbb{R}$.The mapping  $\Phi$ is $\mathcal{D}$-pullback asymptotically compact in $H^{\alpha}$,  as shown by Lemma \ref{compact}. By applying the abstract result \cite[Proposition 2.5]{AG18}, we can obtain the existence and uniqueness of $\mathcal{D}$-pullback random attractor $\mathcal{A} \in \mathcal{D}$ associated with $\Phi$.

\section{Acknowledgments}
 The research of Y. Zhang is supported by the Natural Science Foundation of Henan Province of China (Grant No. 232300420110).

\end{document}